\documentclass[12pt]{amsart}
\numberwithin{equation}{section}

\usepackage{amsfonts, amsmath, amssymb, amsgen, amsthm, amscd,
latexsym}
\usepackage{color}
\usepackage[all]{xy}

\newtheorem{thm}{Theorem}[section]

\newtheorem{cor}[thm]{Corollary}
\newtheorem{lem}[thm]{Lemma}

\newtheorem{rem}[thm]{Remark}

\def\Diff{\mathop{\rm Diff}\nolimits}

\def\GL{\mathop{\rm GL}\nolimits}

\def\O{\mathop{\rm O}\nolimits}

\def\Id{\mathop{\rm Id}\nolimits}

\def\det{\mathop{\rm det}\nolimits}

\def\cop{\mathop{\rm cop}\nolimits}

\newcommand{\ra}{\rightarrow}
\newcommand{\fl}{\forall}
\newcommand{\wt}{\widetilde}
\newcommand{\wg}{\wedge}

\newcommand{\s}{\sigma}
\newcommand{\vars}{\varsigma}
\newcommand{\D}{\Delta}

\newcommand{\Zb}{\mathbb{Z}}

\newcommand{\ot}{\otimes}

\newcommand{\Hc}{\mathcal{H}}
\newcommand{\g}{\gamma}
\newcommand{\vp}{\varphi}
\newcommand{\ve}{\varepsilon}

\newcommand{\Cb}{\mathbb{C}}

\newcommand{\FA}{\mathfrak{A}}

\newcommand{\FN}{\mathfrak{N}}

\newcommand{\Fa}{\mathfrak{a}}
\newcommand{\Fg}{\mathfrak{g}}
\newcommand{\Fh}{\mathfrak{h}}

\newcommand{\Fl}{\mathfrak{l}}

\newcommand{\Fs}{\mathfrak{s}}

\newcommand{\Fo}{\mathfrak{o}}

\newcommand{\p}{\partial}

\def\tb{{\bf t}}

\def\Gb{{\bf G}}

\def\Nbo{{\bf N}}

\def\sb{{\bf s}}
\def\0b{{\bf 0}}

\def\Cb{{\mathbb C}}

\def\Nb{{\mathbb N}}
\def\Rb{{\mathbb R}}

\def\Zb{{\mathbb Z}}

\def\Ac{{\mathcal A}}
\def\Bc{{\mathcal B}}
\def\Cc{{\mathcal C}}
\def\Dc{{\mathcal D}}
\def\Ec{{\mathcal E}}
\def\Fc{{\mathcal F}}
\def\Gc{{\mathcal G}}
\def\Hc{{\mathcal H}}
\def\Ic{{\mathcal I}}

\def\Lc{{\mathcal L}}

\def\Qc{{\mathcal Q}}

\def\Uc{{\mathcal U}}
\def\Vc{{\mathcal V}}

\def\a{\alpha}

\def\d{\delta}
\def\k{\kappa}
\def\lb{\lambda}
\def\g{\gamma}
\def\om{\omega}
\def\s{\sigma}
\def\t{\theta}
\def\ve{\varepsilon}
\def\vp{\varphi}

\def\vp{\varphi}
\def\vars{\varsigma}

\def\D{\Delta}
\def\G{\Gamma}

\def\Om{\Omega}

\def\fl{\forall}
\def\ify{\infty}

\def\nb{\nabla}

\def\ot{\otimes}
\def\part{\partial}

\def\ts{\times}
\def\wdg{\wedge}

\def\ra{\rightarrow}

\def\text{\hbox}

\def\Diff{\mathop{\rm Diff}\nolimits}

\def\End{\mathop{\rm End}\nolimits}

\def\circint{{\oint}}

\def\fl{\forall}
\def\ify{\infty}

\def\nb{\nabla}
\def\ot{\otimes}

\def\ra{\rightarrow}

\def\rt{\rtimes}
\def\lt{\triangleleft}

\def\rt{\triangleright}

\def\lt{\triangleleft}

\def\cl{\blacktriangleright\hspace{-4pt} < }
\def\al{>\hspace{-4pt}\vartriangleleft}

\def\acl{\blacktriangleright\hspace{-4pt}\vartriangleleft }

\def\p{\partial}

\def\wt{\widetilde}
\def\td{\tilde}

\def\0D{\Delta^{(0)}}
\def\1D{\Delta^{(1)}}

\def\dbo{{\bf d}}

\def\tb{{\bf t}}

\def\Gb{{\bf G}}

\def\0b{{\bf 0}}
 
\def\build#1_#2^#3{\mathrel{
\mathop{\kern 0pt#1}\limits_{#2}^{#3}}}
\newcommand{\ps}[1]{~\hspace{-4pt}_{^{(#1)}}}
\newcommand{\ns}[1]{~\hspace{-4pt}_{_{{<#1>}}}}

\newcommand{\lu}[1]{\;^{#1}\hspace{-2pt}}

\def\odots{\ot\dots\ot}

\def\odots{\ot\cdots\ot}
\def\wdots{\wedge\dots\wedge}

\def\one{{\bf 1}}
\def\odots{\ot\dots\ot}
 
\def\dbar{{^-\hspace{-7pt}\d}}
\def\gbar{{-\hspace{-9pt}\g}}
\def\etabar{{-\hspace{-7pt}\eta}}
\def\wdots{\wedge\dots\wedge}
\def\one{{\bf 1}}

\def\0D{\Delta^{(0)}}
\def\1D{\Delta^{(1)}}
\def\Db{\blacktriangledown}

\def\dt{\left.\frac{d}{dt}\right|_{_{t=0}}}

\title[Hopf cyclic classes]{Geometric construction of Hopf cyclic characteristic classes}

\begin{document}
%%%%%%%%%%%%%%%%%%%%%%%%%%%%%%%%%%%%%%%%%%%%%%%
\author{Henri Moscovici}
 \address{Department of mathematics,  
The Ohio State University, 
Columbus, OH 43210, USA}
\email{henri@math.osu.edu}


\thanks{Supported in part by the National Science Foundation
award DMS-1300548}

%%%%%%%%%%%%%%%%%%%%%%%%%%%%%%%%%%%
\begin{abstract}

In earlier joint work with A. Connes on transverse index theory on foliations,
cyclic cohomology adapted to Hopf algebras has emerged as a decisive tool 
in deciphering the total index class of the hypoelliptic signature operator.
We have found a Hopf algebra $\Hc_n$,
playing the role of a `quantum structure group' for the  
 `space of leaves' of a codimension n foliation, whose 
Hopf cyclic cohomology is
 canonically isomorphic to the Gelfand-Fuks
 cohomology of the Lie algebra of formal vector fields.
However, with a few low-dimensional exceptions, no explicit construction
was known for its Hopf cyclic classes.
This paper provides an effective method 
for constructing the Hopf cyclic cohomology classes
of $\Hc_n$ and of $\Hc_n$ relative to $\O_n$,
in the spirit of the Chern-Weil theory, which completely elucidates their 
relationship with the characteristic classes of foliations. 
  \end{abstract}

 \maketitle
 
 \tableofcontents
 
 \section*{Introduction}

In our joint work with A. Connes 
on the local index formula for transversely elliptic operators on foliations~\cite{CM98},
a certain Hopf algebra $\Hc_n$ turned out to play
the role of a `quantum structure group' for the 
 `space of leaves' of any codimension $n$ foliation. Moreover, the
 Hopf-version of cyclic cohomology, which emerged in the same paper,
was shown for $\Hc_n$ to be
 canonically isomorphic via a quasi-isomorphism of van Est type
 to the Gelfand-Fuks
 cohomology of the Lie algebra of formal vector fields on $\Rb^n$.
This isomorphism furnished the decisive tool
in relating the total index class of the hypoelliptic signature operator~\cite{CM95} to
the characteristic classes of foliations. 
Furthermore, the transplantation of
the Gelfand-Fuks classes in the Hopf cyclic cohomological framework
broadened the scope of their applicability, as illustrated by the 
work on modular Hecke algebras~\cite{CM04}, which gave a `modular' interpretation
to the basic Hopf cyclic cocycles of $\Hc_1$. 
However, apart from $\Hc_1$ (cf. \cite{CM98, MR07}), explicit cocycle
representatives for all Hopf cyclic cohomology classes  
were known only for $\Hc_2$ (cf. \cite{RanSut}).
\smallskip

The present paper provides a geometric method for representing
the Hopf cyclic cohomology classes of $\Hc_n$, and of $\Hc_n$ relative to $\O_n$,
by concrete cocycles, in the spirit of Chern-Weil theory. Besides giving
an effective construction of the Hopf cyclic characteristic classes, this procedure
renders their 
relationship with the characteristic classes of foliations completely transparent. 
\smallskip

In addition to ideas and results
from~\cite{CM98} as well as their subsequent refinements obtained in collaboration with 
B. Rangipour (\cite{MR07, MR09, MR11}), our approach uses as 
key additional ingredients the `differentiable' modification,
defined \`a la Haefliger,
of the Bott bicomplex~\cite{Bott*, BSS} for equivariant cohomology
and of Dupont's simplicial de Rham 
DG-algebra~\cite{Dupont}. 

In their standard version the above complexes
compute the $\Diff (M)^\d$-equivariant cohomology of a manifold. We first show that 
their differentiable counterparts deliver Haefliger's 
differentiable cohomology~\cite{HaefDC}
of the \'etale groupoid associated to the tautological action of $\Diff (M)^\d$,
and thus the geometric characteristic classes of foliations.
We next prove that the quasi-isomorphism of van Est type constructed in~\cite{CM98}
transits through the differentiable  simplicial de Rham DG-algebra before
landing in the differentiable Bott complex.
The former being graded commutative, its cohomology classes can be constructed
by the usual Chern-Weil procedure ~\cite{Cartan}. 
The transition to the Bott complex is effected by 
integration along the fibers. Although not multiplicative, this operation provides
a quasi-isomorphism which offers the advantage
of being explicitly computable. Employing then chain maps (from \cite{CM98} and \cite{MR11}), we transfer the representative cocycles, constructed in terms of connection and curvature,
from the differentiable Bott complex to the
original cyclic model~\cite{CM98} for the Hopf cyclic cohomology of  $\Hc_n$
as well as to the quasi-isomorphic model of Chevalley-Eilenberg type
constructed in~\cite{MR11}. 
\smallskip

The upshot is a concrete construction of bases
for the Hopf cyclic cohomology of $\Hc_n$ and also for $\Hc_n$ relative to $\O_n$,
in both cohomological models mentioned above,
on a par with the classical geometric
construction of characteristic classes of foliations~\cite{Bott-LNM, BottHaef, KambTond}.
In particular this construction provides
``minimal'' representative cocycles for all Hopf cyclic cohomology classes,
reproducing the known feature of the Gelfand-Fuks cohomology 
of being representable by
cocycles involving jets of order no higher than two of the formal vector fields.

\section{Characteristic cocycles in differentiable cohomology} \label{ExpCoc}

  \subsection{Differentiable equivariant cohomology} \label{DEC}
  
Let $M$ be a smooth oriented manifold of dimension $n$, and let
$\Gb = \Diff (M)$ be its group of diffeomorphisms. 
Regarding $\Gb$ as a discrete group,
the equivariant cohomology $H^\bullet_{\Gb} (M, \Rb)$, originally defined by
means of the homotopy quotient as $H^\bullet( E\Gb \ts_{\Gb} M, \Rb)$, 
can be expressed in terms of de Rham complexes. These are
associated to the simplicial manifold 
\begin{align*} 
\triangle_{\Gb} M = \{\triangle_{\Gb} M[p] := 
 \Gb^p \ts M\}_{p\geq 0} \, ,
\end{align*}
with face maps $\p_i : \triangle_{\Gb} M[p] \ra \triangle_{\Gb} M[p-1]$,  \, $ 1 \leq i \leq p$, \,  given by 
\begin{equation*} 
\p_i (\phi_1, \ldots , \phi_p, \, x) = \left\{ \begin{matrix}
&(\phi_2, \ldots , \phi_p,\,  x) \, , \quad &i=0 \, , \\
&(\phi_1, \ldots , \phi_i \phi_{i+1}, \ldots , \phi_p, x) \, , \quad &1 < i < p \, , \\
&(\phi_1, \ldots , \phi_{p-1}, \, \phi_p (x)) \, , \quad &i=p \, .
\end{matrix} \right.
\end{equation*} 
 and degeneracies
$$
\s_i (\phi_1, \ldots , \phi_p, \, x) = (\phi_1, \ldots , \phi_i , e , \phi_{i+1},\ldots , \phi_p, x) \, ,
\quad 0 \leq i \leq p .
$$
 
The first such complex (cf. \cite{Bott*, BSS}) is the total complex of
the bicomplex
$\{C^\bullet \left(\Gb, \Om^\bullet (M)\right), \d, d \}$ defined as follows: 
$C^p \left(\Gb, \Om^q (M)\right)$ is spanned by cochains  
\begin{align*} 
c (\phi_1, \ldots , \phi_p) \in \Om^q (M) , \qquad  \phi_1, \ldots , \phi_p \in \Gb ,
\end{align*}
$d$ is the de Rham differential, and $\d$ is the  group cohomology boundary  
 \begin{align*} 
 \begin{split}
 \d c (\phi_1, \ldots , \phi_{p+1})\,  = & \sum_{i=0}^{p} (-1)^i  
 c \big(\p_i (\phi_1, \ldots , \phi_{p+1})\big) \\
 & + (-1)^{p+1} 
 \phi_{p+1}^* c (\phi_1, \ldots , \phi_p) .
 \end{split}
 \end{align*}  

Instead of the action groupoid notation implicitly used in the above formulas 
it will be more convenient to work with the homogeneous bicomplex 
$\{\bar{C}^\bullet \left(\Gb, \Om^\bullet (M)\right), \bar{\d}, d \}$. Its 
$(p, q)$-cochains 
\begin{align*} 
\bar{c} (\rho_0 , \ldots , \rho_{p}) \in \Om^q (M) , \qquad  \rho_0 , \ldots , \rho_{p} \in \Gb ,
\end{align*}
satisfy the covariance condition 
\begin{align} \label{cov1}
\bar{c} (\rho_0 \rho, \ldots , \rho_{p} \rho) = \rho^* \bar{c} (\rho_0, \ldots , \rho_p) ,
\quad \forall \, \rho, \rho_i \in \Gb ,
\end{align}
and the group cohomology boundary is given by
 \begin{equation*} 
 \bar{\d} \bar{c} (\rho_0, \ldots , \rho_{p}) =
  \sum_{i=0}^{p} (-1)^i \bar{c} (\rho_0, \ldots , \check{\rho_i},
 \ldots , \rho_{p}) , 
\end{equation*} 
the `check' mark signifying the omission of the element underneath.

The passage between the two isomorphic bicomplexes is via the relations
 \begin{align} \label{xchng}
 \begin{split}
  &c (\phi_1, \ldots , \phi_p) = \bar{c} (\phi_1 \cdots \phi_p ,\,  \phi_2 \cdots \phi_p ,\,
  \ldots , \phi_p, e) \, , \\
  \text{resp.} \quad 
 &\bar{c}(\rho_0, \ldots , \rho_p) = \rho_p^* c(\rho_0 \rho_1^{-1}, \rho_1 \rho_2^{-1} ,
 \ldots , \rho_{p-1} \rho_p^{-1}) .
 \end{split}
 \end{align}
 
\smallskip

The second complex computing $H^\bullet_{\Gb} (M, \Rb)$ is Dupont's
de Rham complex (cf. \cite{Dupont}) of compatible forms 
$\{\Om^\bullet (|\triangle_{\Gb} M|), d \}$
on the geometric realization $|\triangle_{\Gb} M|$. By definition,
\begin{align*} 
\Om^\bullet (|\triangle_{\Gb} M|) \subset \prod_{p=0}^\ify \Om^\bullet( \D^p \ts \triangle_{\Gb} M[p])
\end{align*}
 consists of sequences $\om = \{\om_p\}_{p\geq 0}$, with
 $\om_p \in \Om^\bullet(\D^p \ts \triangle_{\Gb} M[p])$, such that for all morphisms
 $\mu \in \D(p, q)$ in the simplicial category, 
 \begin{align} \label{sform}
 (\mu_\bullet \ts \Id)^\ast \om_q \, = \, 
 (\Id \ts \mu^\bullet )^\ast \om_p \,\in  \Om^\bullet \left(\D^p \ts \triangle_{\Gb} M[q]\right) .
\end{align}
Here \, $
\D^p = \{\tb = (t_0, \ldots , t_p) \in \Rb^{p+1} \, \mid \, t_i \geq 0, \quad t_0 + \ldots + t_p =1\} $,
$\mu_\bullet  : \D^p \ra \D^q$ (resp. $\mu^\bullet  : \triangle_{\Gb} M[q] \ra \triangle_{\Gb} M[p] $),
stands for the induced cosimplicial (resp. simplicial) map, and 
$\Om^k(\D^p \ts \triangle_{\Gb} M[q])$ denotes the $k$-forms on $\D^p \ts \triangle_{\Gb} M[q] $ which
are extendable to smooth forms on $V^p  \ts \triangle_{\Gb} M[q]$, where
$V^p = \{\tb = (t_0, \ldots , t_p) \in \Rb^{p+1} \, \mid \,  t_0 + \ldots + t_p =1\} $.
By~\cite[Thm 2.3]{Dupont}, the operation of integration along along the fibers
\begin{align} \label{circint}
\oint_{\D^p}: \Om^{\bullet} (\D^p \ts \triangle_{\Gb} M[p]) \ra 
 \Om^{\bullet -p} (\triangle_{\Gb} M[p])  \,
\end{align}
establishes a quasi-isomorphism
between the complexes  $\{\Om^\bullet (|\triangle_{\Gb} M|), d \}$ and
$\{ C^{\rm tot \bullet} \left(\Gb, \Om^\ast (M)\right), \d \pm d \}$. 
 
 As in the case of the Bott complex, there is a homogeneous description of the
 simplicial de Rham complex, $\{\Om^\bullet ({|\bar\triangle}_{\Gb} M|), d \}$,
 consisting of the $\Gb$-invariant compatible forms on the geometric realization 
  $|\bar\triangle_{\Gb} M|$. The simplicial manifold $\bar\triangle_{\Gb} M$ 
  is defined as follows:
\begin{align*} 
\bar\triangle_{\Gb} M = \{\bar\triangle_{\Gb} M[p] := 
 \Gb^{p+1} \ts M\}_{p\geq 0} \, ,
\end{align*}
with face maps 
$\bar\p_i : \bar\triangle_{\Gb} M[p] \ra \bar\triangle_{\Gb} M[p-1]$,  \, $ 1 \leq i \leq p$, \,  given by 
\begin{equation*} 
\bar\p_i (\rho_0, \ldots , \rho_p, \, x) \, = \, (\rho_0, \ldots , \check{\rho_i},
 \ldots , \rho_{p}) ,  \quad 0 \leq i \leq p ,
\end{equation*} 
 and degeneracies
$$
\bar\s_i (\rho_0, \ldots , \rho_p, \, x) = (\rho_0, \ldots , \rho_i , \rho_i, \ldots , \rho_p, x) \, ,
\quad 0 \leq i \leq p .
$$
 The compatible forms $\om = \{\om_p\}_{p\geq 0} \in \Om^\bullet ({|\bar\triangle}_{\Gb} M|$
 satisfy the invariance condition
  \begin{align} \label{cov2}
 \om (\rho_0, \ldots , \rho_p) = (\rho^{-1})^*\om (\rho_0 \rho, \ldots , \rho_{p} \rho) ,
\quad \forall \, \rho, \rho_i \in \Gb .
\end{align}
\smallskip

For the purposes of this paper, the relevant cohomology is
the \textit{differentiable} modification of the above constructs, 
in the sense of Haefliger~\cite[Ch.4, \S4]{HaefDC}. 

In the case of the Bott complex the modification amounts to pass to 
the subcomplex of differentiable cochains
$\{C_{\rm d}^{\rm tot \bullet} \left(\Gb, \Om^\ast (M)\right), \d \pm d \}$,
resp. $\{ \bar{C}_{\rm d}^{\rm tot \bullet} \left(\Gb, \Om^\ast (M)\right), \d \pm d \}$,
which is defined as follows:
a cochain $\om \in \bar{C}^p \left(\Gb, \Om^q (M)\right)$ 
is {\em differentiable} 
if for any local chart $U \subset M$ with coordinates $(x^1, \ldots , x^n)$,
\begin{align} \label{difco}
\om (\rho_0, \ldots , \rho_p, x) = 
\sum f_I \left(x, j^k_x(\rho_0), \ldots , j^k_x(\rho_p) \right) dx^I ,
\end{align}
with  $f_I$ smooth functions of $x \in U$ and the $k$-jets at $x$ of 
$\rho_0, \ldots , \rho_p$, for some $k \in \Nb$, and $dx^I = dx^{i_1} \wg \ldots \wg dx^{i_q}$
with $I = (i_1< \ldots < i_q)$ running through the set of strictly increasing $q$-indices.
 
The cohomology of the total complex 
$\{ \bar{C}_{\rm d}^{\rm tot \bullet} \left(\Gb, \Om^\ast (M)\right), \d \pm d \}$ 
will be denoted  $H_{\rm d, \Gb}^{\bullet}\left(M, \Rb \right)$.
\smallskip

In the case of Dupont's complex, the 
{\em differentiable simplicial de Rham complex } 
is the subcomplex $\{\Om_{\rm d}^\bullet (|\bar\triangle_{\Gb} M|), d \}$ of
$\{\Om^\bullet (|\bar\triangle_{\Gb} M|), d \}$  
consisting of the $\Gb$-invariant compatible forms
$\{\om_p\}_{p\geq 0}$ whose components 
satisfy the analogous condition:
 \begin{align} \label{diffo}
 \begin{split}
\om_p (\tb; \rho_0, \ldots , \rho_p, x) = 
\sum f_{I, J} \left(\tb; x, j^k_x(\rho_0), \ldots , j^k_x(\rho_p) \right) dt^I \wg dx^J ,
 \end{split}
\end{align}
with $f_{I, J}$ smooth in all variables.
 
We denote by $H_{\rm d}^\bullet (|\triangle_{\Gb} M|, \Rb)$ the cohomology of the
differentiable simplicial de Rham complex 
$\{\Om_{\rm d}^\bullet (|\bar\triangle_{\Gb} M|), d\}$.

\begin{thm} \label{difDup}
The chain map $\displaystyle \circint_{\D^\bullet} : \Om_{\rm d}^\bullet (|\bar\triangle_{\Gb} M|) \ra 
\bar{C}_{\rm d}^{\bullet} \left(\Gb, \Om^\ast (M)\right)$ induces an isomorphism
$H_{\rm d}^\bullet (|\triangle_{\Gb} M|, \Rb) \, \cong H_{\rm d, \Gb}^{\bullet}\left(M, \Rb \right)$.
\end{thm}

\begin{proof}
 Clearly, the integration along the fibers maps
$\Om_{\rm d}^\bullet (|\bar\triangle_{\Gb} M|)$ to
$\bar{C}_{\rm d}^{\bullet} \left(\Gb, \Om^\ast (M)\right)$. Moreover, the
natural chain maps in both directions as well as 
the chain homotopies relating them in the proof of Theorem 2.3 in~\cite{Dupont}
preserve the differentiable subcomplexes.  
\end{proof}

 \smallskip
 
  \subsection{Explicit van Est-Haefliger isomorphism} \label{ExpF}
 
We denote by $F^{k}M$ the frame bundle of order $k \in \Nb \cup \infty$,
 formed of $k$-jets at $0$ of local diffeomorphisms $\phi$ from a neighborhood 
 of $0 \in \Rb^n$ to a neighborhood of $\phi(0) \in M$. 
 Thus, $F^{1}M = FM$ is the usual
 frame bundle, while by definition $F^{\ify}M := \varprojlim F^{k}M$.
 Each $F^{k}M$  is a principal
 bundle over $M$ with structure group $\Gc^k$ formed of $k$-jets at $0$ of 
 local diffeomorphisms of $\Rb^n$ preserving $0$.  
The group $\Gb$ operates
on the left on each $F^{k}M$ by:
 $$
 \phi \cdot j_0^\ify (\rho) := j_0^\ify (\phi \circ \rho) , \qquad \phi \in \Gb, \, \rho \in F^{k}M .
$$
\smallskip
 
Let now $\Fa_n$ be Lie algebra of formal vector fields on $\Rb^n$.
 Any $\,v \in \Fa_n $ can be represented as 
 $\displaystyle \quad v = j_0^\infty \left(\frac{d}{dt} \mid_{t=0} \rho_t\right)$,
 with $\{\rho_t\}_{t \in \Rb}$ a $1$-parameter group of local diffeomorphisms of $\Rb^n$;
 it thus gives rise to a $\Gb$-invariant vector field on $F^{\ify}M$, defined at a point
 $j_0^\infty (\phi) \in F^{\ify}M$ by
$$
 \tilde{v} \mid_{j_0^\infty (\phi)}  = j_0^\infty \left(\frac{d}{dt} \mid_{t=0} (\phi \circ \rho_t ) \right).
 $$
Dually, any $\, \om  \in C^{m}(\Fa_n)$, where
$C^{\ast}(\Fa_n)$ denotes the Gelfand-Fuks cohomology complex \cite{GF} of $\Fa_n$,
 gives rise to a  $\Gb$-invariant
 form $\tilde\om \in \Om^m (F^{\ify}M)$, characterized by
$$
 \tilde\om (\tilde{v}_1 , \ldots , \tilde{v}_m) = \om (v_1, \ldots , v_m) .
 $$ 
Moreover, the assignment 
\begin{equation} \label{invid}
\om \in C^\bullet(\Fa_n) \mapsto \tilde\om \in \Om^\bullet (F^{\ify}M)^\Gb
\end{equation}
is a DGA-isomorphism, by means of which we shall tacitly identify the two DG-algebras.
 \smallskip
  
Choosing a torsion-free affine connection $\nabla$ on $M$, one defines 
a cross-section $\s_{\nabla} : FM \ra F^{\ify}M$
of the natural projection $\pi_1 : F^\ify M \ra FM$,
by the formula
\begin{equation} \label{jcon}
\s_{\nabla} (u) = j_0^{\ify} (\exp_x^{\nabla} \circ u) \ , \qquad u \in F_x M \, .
\end{equation}
This cross-section is clearly ${\rm GL}_n$-equivariant
\begin{equation} \label{Rinv}
\s_{\nabla} \circ R_a = R_a \circ \s_{\nabla} \, , \qquad 
 a \in {\rm GL}_n  \, ,
\end{equation}
as well as Diff-equivariant, 
\begin{equation} \label{nat}
\s_{\nabla^\phi} = {\phi}^{-1} \circ \s_{\nabla}  \circ {\phi} \ , 
\qquad \fl \, \phi \in \Gb  \, .
\end{equation}
Here $\nabla^\phi = \phi_\ast^{-1} \circ \nabla \circ \phi_\ast$, or more precisely
the derivative whose connection form is the pull-back ${\phi}^* (\om_\nabla)$ of
the connection form of $\nabla$.
 \smallskip
 
Let $\bar\triangle_{\Gb} FM$ be the  
simplicial manifold associated to the action of $\Gb$ by prolongation on $FM$.
We define the maps $\s_p : \D^p \ts \bar\triangle_{\Gb} FM[p] \ra F^\ify M$, $p \in \Nb$, 
by the formula
\begin{align} \label{xchng2}
\begin{split}
\s_p (\tb ; \rho_0 , \ldots , \rho_p, u) &= \,
\s_{\nabla (\tb; \rho_0 , \ldots , \rho_p)} (u)  ,\\
\text{where}   \quad
 \nabla (\tb; \rho_0 , \ldots , \rho_p) &= \, \sum_{i=0}^p \ t_i \, \nabla^{\rho_i} , \qquad
 \tb \in \Delta^p .
\end{split}
\end{align}
Manifestly, the collection $\, \hat{\s} = \{\s_p \}_{p \geq 0}$ 
descends to the geometric realization of
 $\bar\triangle_{\Gb} FM$, yielding a well-defined map 
 $\, \hat{\s} : |\bar\triangle_{\Gb} FM| \ra F^\ify M$; moreover, 
 this map is $\GL_n$-equivariant.

 \begin{thm} \label{main1}
 {\rm \bf (1)} Let $\om \in C^\bullet(\Fa_n)$ then
 $\hat{\s}^*(\tilde\om) \in \Om_{\rm d}^\bullet (|\bar\triangle_{\Gb} FM|)$, and the map
 $\,  \Cc_\nabla :  C^\bullet(\Fa_n) \ra \Om_{\rm d}^\bullet (|\bar\triangle_{\Gb} FM|)$,
 defined by
\begin{align} \label{crux}
 \Cc_\nabla (\om )\, = \, 
  \hat{\s}^*(\tilde\om) \in \Om_{\rm d}^\bullet (|\bar\triangle_{\Gb} FM|) ,
\end{align}
 is a quasi-isomorphism of DG-algebras.
 
{\rm \bf (2)}  The map $\, \Cc_\nabla$ is  $\GL_n$-equivariant 
and, by restriction to the subcomplex of \, $\O_n$-basic cochains,
it induces a quasi-isomorphism of DG-algebras  
 $\, \Cc^{\O_n}_\nabla: C^\bullet(\Fa_n, \O_n) \ra 
 \Om_{\rm d}^\bullet  (|\bar\triangle_{\Gb} (PM, \O_n)|)$; here $PM = FM/\O_n$
 and $\O_n$-basic forms on $FM$ are identified with forms on $PM$.
  \end{thm}
 
 \begin{proof} Since $\tilde\om$ is $\Gb$-invariant, $\, \hat{\s}^*(\tilde\om)$ is 
indeed a compatible form. It is also quite obvious that it belongs to the differentiable
subcomplex $\, \Om_{\rm d}^\bullet (|\bar\triangle_{\Gb} FM|)$. 
Furthermore,
 $\, \hat{\s}^*(d\tilde\om) =   d \big(\hat{\s}^*(\tilde\om)\big)$, and so
$\,  \Cc_\nabla$ is a well-defined map of complexes.
 
 Observe now that for any connection $\td\nabla$,
 \begin{equation*} 
(\pi_1 \circ \s_{\td\nabla}) (u) = j_0^{1} (\exp_x^{\td\nabla} \circ u) = u, \quad u \in F_x M \, .
\end{equation*}
Upgrading both $\pi_1$ and $\hat{\s}$ in the obvious way to simplicial maps
$ \Id \ts \pi_1 : |\bar\triangle_{\Gb} F^\ify M| \ra |\bar\triangle_{\Gb} FM|$ and
$\Id \ts \hat{\s}  : |\bar\triangle_{\Gb} FM| \ra |\bar\triangle_{\Gb} F^\ify M|$,
it follows that
\begin{align*} 
(\Id \ts \pi_1) \circ (\Id \ts \hat{\s})  \, =\, \Id .
\end{align*}
Hence $ (\Id \ts \hat{\s})* : \Om_{\rm d}^\bullet (|\bar\triangle_{\Gb} F^\ify M|) \ra
 \Om_{\rm d}^\bullet (|\bar\triangle_{\Gb} FM|)$
is a left inverse for 
$(\Id \ts \pi_1)^* : \Om_{\rm d}^\bullet (|\bar\triangle_{\Gb} FM|) \ra 
\Om_{\rm d}^\bullet (|\bar\triangle_{\Gb} F^\ify M|)$. Because the fibers of $\pi_1$ are
contractible, $(\Id \ts \pi_1)^*$ induces an isomorphism in (differentiable) cohomology.
Therefore so does its inverse $ (\Id \ts\hat{\s})^* $.
 
 On the other hand, the usual horizontal homotopy (cf.~\cite[\S IV.4]{HaefDC}),
  defined by the formula
 \begin{align*}
 \begin{split}
 (H\a)_{p-1}&(\tb; \rho_0, \ldots, \rho_{p-1} , j_0^\ify(\rho)) = \\
&\a_p (\tb; \rho^{-1}, \rho_0, \ldots, \rho_{p-1} , j_0^\ify(\rho)), \qquad
 \a \in  \Om_{\rm d}^\bullet (|\bar\triangle_{\Gb} F^\ify M| ,
 \end{split}
 \end{align*}
 shows that the natural inclusion
 of $\Om^\bullet (F^\ify M)^\Gb$ into  $\Om_{\rm d}^\bullet (|\bar\triangle_{\Gb} F^\ify M|)$ 
is also quasi-isomorphism. 
Recalling the identification \eqref{invid}, the proof is achieved by noting that
 when restricted to $\Om^\bullet (F^\ify M)^\Gb$ the map
 $ (\Id \ts\triangle \s )^*$ coincides with $\Cc_\nabla$.
\smallskip
 
 The second claim has a similar proof. Identifying
 the $\O_n$-basic forms on $F^\ify M$ with forms on 
 $P^\ify M = F^\ify M/\O_n$, the appropriate homotopy takes the form
  \begin{align*}
 \begin{split}
 (H\a)_{p-1}(\tb; &\rho_0, \ldots, \rho_{p-1} , \, j_0^\ify(\rho) \O_n) = \\
&\int_{\O_n} \a_p (\tb; k^{-1}\rho^{-1} , \rho_0, \ldots, \rho_{p-1} , \, j_0^\ify(\rho) \O_n) dk .
 \end{split}
 \end{align*}
 \end{proof}

Combining the above theorem with Theorem \ref{difDup},
one obtains the following 
explicit form of the van Est-Haefliger isomorphism~\cite[\S IV.4]{HaefDC}.

\begin{thm} \label{main2}
 The maps   $\displaystyle \Dc_\nb = \circint_{\D^\bullet}  \Cc_\nabla
 : {C}^\bullet(\Fa_n) \ra \bar{C}_{\rm d}^{\rm tot \, \bullet} \left(\Gb, \Om^\ast (FM)\right)$,
  resp. $\displaystyle  \Dc^{\O_n}_\nabla = \circint_{\D^\bullet} \Cc^{\O_n}_\nabla
 : C^\bullet(\Fa_n, \O_n) \ra 
\bar{C}_{\rm d}^{\rm tot \, \bullet} \left(\Gb, \Om^\ast (PM)\right)$, 
are quasi-isomorphisms of complexes. 
\end{thm}

\smallskip 

\subsection{Characteristic cocycles} \label{CharCo}

Although explicit, the map $\Dc_\nb$ is quite intricate and thus not
amenable to concrete computations.
Instead, we proceed now to describe an alternative construction of the
$\Diff$-equivariant geometric characteristic classes, in the spirit of the Chern-Weil theory (cf. \cite{Cartan}),
in terms of cocycles manufactured out of the connection and curvature forms.  
 \smallskip
   
 The universal connection and curvature forms 
 $\vartheta = (\vartheta^i_j)$ and $R = (R^i_j)$, defined as in
 \cite[\S 2]{Bott*}, generate
 a DG-subalgebra $CW^\bullet (\Fa_n)$ of $C^\bullet (\Fa_n)$.   
 By the Gelfand-Fuks theorem (cf.~\cite{GF, Godbillon}), the inclusion 
 $CW^\bullet (\Fa_n) \hookrightarrow C^\bullet (\Fa_n)$ is a quasi-isomorphism.
 Actually, there is a faithful embedding of the
truncated Weil complex  $\hat{W}(\Fg\Fl_n)$ which identifies it with the
subcomplex $CW^\bullet (\Fa_n) $ of $C^\bullet (\Fa_n) $.
We recall that $\,\hat{W}(\Fg\Fl_n) = W(\Fg\Fl_n)/ \Ic_{2n}$,  \,where
$\, W(\Fg\Fl_n) = \wg^\bullet \Fg\Fl^*_n \ot S(\Fg\Fl_n)$  is the Weil algebra
of $\Fg\Fl_n$, and
$\Ic_{2n}$ is  the ideal generated by the elements of $S(\Fg\Fl_n) $ of degree $> 2n$.
 These DG-algebras are $\GL_n$ algebras as well. Let
  $CW^\bullet (\Fa_n, \O_n)$, resp. $\hat{W}(\Fg\Fl_n,  \O_n)$, denote
  their subalgebras consisting of $ \O_n$-basic elements. The above identification
 $\hat{W}(\Fg\Fl_n) \equiv CW^\bullet (\Fa_n)$ then restricts to an identification
  $\hat{W}(\Fg\Fl_n,  \O_n) \equiv CW^\bullet (\Fa_n ,  \O_n)$. 

With $\nb$ being as before a fixed torsion-free connection and
 $  \om_\nb = (\om_j^i)$, resp.  $\Om_\nb = (\Om_j^i )$, 
 denoting its matrix-valued connection,
 resp. curvature form on $FM$, one has the naturality relation: 
 
\begin{lem}\label{Luc}
\quad $\s_{\nabla}^* (\wt{\vartheta}_j^i) = \om_j^i$ \, and \quad
$\s_{\nabla}^* (\wt{R}_j^i) = \Om_j^i$.
\end{lem}

 \begin{proof} (Cf.~\cite[Lemma 18]{CM01}.)  Since
\begin{equation} \label{curv}
R_j^i = d \, \vartheta_j^i + \vartheta_k^i \wedge \vartheta_j^k \, ,
\end{equation}
the second identity is a consequence of the first. 
To prove the first, we note that by (\ref{nat}) the  
operator $\om_{\nabla} \mapsto \s_{\nabla}^* (\wt{\vartheta})$, acting   
on the (affine) space of torsion-free 
connections on $FM$, is natural, i.e. $\Gb$-equivariant. 
The uniqueness of such operators on
torsion-free connections (cf.~\cite[\S 25.3]{K-M-S}) ensures that 
the only such operator is the identity. 
 \end{proof}

  In homogeneous group coordinates (see \eqref{xchng}), the
  simplicial connection form-valued matrix
   $\hat{\om_\nb} = \{ \hat\om_p \}_{p \in \Nb}$ associated to $\nabla$
  has components
\begin{align} \label{scone}
 \hat\om_p (\tb ; \rho_0, \ldots , \rho_p) : = \sum_{i=0}^p t_i \rho_i^* (\om_\nb) ,  
\end{align}
and the 
 simplicial matrix-valued curvature form
 $\hat{\Om}_\nb : = d \hat{\om}_\nb + \hat{\om}_\nb \wg \hat{\om}_\nb$
has components \, $\hat{\Om}_p = \hat{\Om}_p^{(1, 1)} + \hat{\Om}_p^{(0, 2)}$,
given by
\begin{align}  \label{scurv}
\begin{split}
&\hat{\Om}_p (\tb ; \rho_0, \ldots , \rho_p) \, = \, \sum_{i=0}^p dt_i \wdg \rho_i^* (\om_\nb) \, + \\
&\sum_{i=0}^p t_i \big(\rho_i^* (\Om_\nb) - \rho_i^* (\om_\nb) \wdg \rho_i^* (\om_\nb)\big) 
+ \sum_{i, j=0}^p t_i t_j \,  \rho_i^* (\om_\nb) \wdg  \rho_j^* (\om_\nb) .
\end{split}
 \end{align} 
 The forms $\hat{\om}_j^i$ and $\hat{\Om}_j^i$ clearly
belong to the differentiable de Rham complex $\Om_{\rm d}^\bullet (|\bar\triangle_{\Gb} FM|)$.
\smallskip

In view of the above discussion, Theorem \ref{main1} 
together with Lemma \ref{Luc} have the following consequence.

\begin{cor} \label{UCW1} 
{\rm  \bf (1a)} The forms  $\, \hat{\om}^i_j$ and  $\, \hat{\Om}^i_j$ generate
 a DG-subalgebra $CW_{\rm d}^\bullet (|\bar\triangle_{\Gb} FM|)$, and
 $\Cc_\nb$ restricts to an isomorphism of   $CW^\bullet (\Fa_n) \equiv \hat{W}(\Fg\Fl_n)$
  onto  $CW_{\rm d}^\bullet (|\bar\triangle_{\Gb} FM|)$. 
 
 {\rm  \bf (2a)}  By restriction to the respective DG-subalgebras of $ \O_n$-basic elements, 
 $\Cc_\nb$ induces an isomorphism of
 $CW^\bullet (\Fa_n,  \O_n)) \equiv \hat{W}(\Fg\Fl_n,  \O_n) $
onto $CW_{\rm d}^\bullet (|\bar\triangle_{\Gb} PM|)$.
 \end{cor}

The operation of integration along the fibers does not preserve the 
 cup product (which is graded commutative at the source but not in the target).
 Nevertheless, by Theorem \ref{main2}, it still induces isomorphism in 
 cohomology.
 
\begin{cor} \label{UCW2}      
{\rm  \bf (1b)}  The restriction of $\Dc_\nabla$  to $CW^\bullet (\Fa_n,  \O_n)$
is a quasi-isomorphism to 
$\{C_{\rm d}^{\rm tot \, \bullet} \left(\Gb, \Om^\ast (FM)\right), \d \pm d \}$.
 
{\rm  \bf (2b)} The restriction of $\Dc^{\O_n}_\nabla$ to  $CW^\bullet (\Fa_n,  \O_n) $
is a quasi-isomorphism to  
$\{C_{\rm d}^{\rm tot \, \bullet} \left(\Gb, \Om^\ast (PM)\right), \d \pm d \}$.
 \end{cor}

To describe concrete bases of cohomology classes constructed in the Chern-Weil manner,
we let $I (\Fg \Fl_n) = S(\Fg \Fl_n)^{\GL_n}$ be the algebra of invariant polynomials,
or equivalently, the subalgebra of $\GL_n$-basic elements of $W (\Fg \Fl_n)$. 
Once the
torsion-free connection $\nabla$ is chosen, to any polynomial
$ \, P \in I (\Fg \Fl_n) $ one associates a closed simplicial
differential form $P(\hat{\Om}) \in \Om_{\rm d}^\bullet (|\triangle_{\Gb} FM|)$.  
On the total space of the frame bundle this form is exact, and 
can be expressed as a boundary
by a standard transgression formula (cf.~\cite{ChernSim}):
\begin{align} \label{TP}
\begin{split}
&P(\hat\Om_\nb) = d (TP(\hat\om_\nb)) , \quad \text{with} \\
&TP(\hat\om_\nb) =   k  \int_0^1 P\left(\hat\om_\nb, \hat{\Om}_t, \ldots , \hat{\Om}_t \right)  dt , 
\quad k = \deg(P),\\
&\text{where} \quad \hat{\Om}_t= t \hat\Om_\nb + (t^2 - t) \hat\om_\nb \wg \hat\om_\nb .
\end{split}
\end{align}

Corollary \ref{UCW2} allows now to transfer a Vey basis~\cite{Godbillon} of $H^* (\Fa_n)$
to a basis of $H_{\rm d} (|\triangle_\Gb FM| , \Rb)$ as follows. Let $\{c_k \}_{1 \leq k \leq n}$ 
be a system of generators of the algebra $I (\Fg \Fl_n)$, for example 
the coefficients of the powers of $t$ in the expansion
\begin{align}
\det \left(\Id - \frac{t}{2\pi i} A\right) \, = \, \sum_{k=1}^n t^k c_k (A), \quad A \in \Fg \Fl_n(\Cb) .
\end{align}
These give the classical Chern forms $c_k (\hat\Om_\nb) \in  \Om_{\rm d}^{2k} (|\triangle_{\Gb} FM|)$,
and by transgression the Chern-Simon forms 
$Tc_k (\hat\om_\nb) \in  \Om_{\rm d}^{2k-1} (|\triangle_{\Gb} FM|)$. 
 \smallskip
 
The image $\displaystyle C_k (\hat\Om_\nb) = \circint_{\D^\bullet} c_k (\hat\Om_\nb)$ is
the cocycle $C_k (\hat\Om_\nb) = \{C_k^{(p)}(\hat\Om_\nb)\}_{p\geq 0}$ whose components
in homogeneous group coordinates are
\begin{align} \label{Chcoc}
 C_k^{(p)} (\hat\Om_\nb) (\phi_0,\ldots, \phi_p) \,= \, 
(-1)^p \circint_{\D^p} c_k\left(\hat\Om_\nb (\tb ; \phi_0, \ldots , \phi_p)\right) .
\end{align}
Similarly, the image 
$\displaystyle  TC_k (\hat{\Om}_\nb) = \circint_{\D^\bullet} Tc_k (\hat{\Om}_\nb)$
 is the transgressed cochain 
 $\, TC_k(\hat{\om}_\nb)= \{TC_k^{(r)}(\hat{\om}_\nb)\} $ with
  homogeneous components given by 
  \begin{align}  \label{TPcoc}
TC_k^{(r)}(\hat{\om}_\nb) (\phi_0,\ldots, \phi_r) \,= \, 
(-1)^p \circint_{\D^p} Tc_k \big(\hat\om_\nb(\tb ; \phi_0, \ldots , \phi_r)\big) .
\end{align}

The Vey basis  can now be transferred 
as follows. Consider the collection $\Vc_n$ of all pairs $(I, J)$ of
 subsets of $\{1, \ldots , n\}$ of the form
 $I = \{i_1 < \ldots < i_p \}$ and $J=\{j_1 \leq \ldots \leq j_q\}$, such that
 $|J| := j_1 +\ldots + j_q \leq n$, $\,   i_1 \leq j_1$ and $\,  i_1 +  |J| > n$.

 \begin{cor} \label{Vbasis}
With $(I, J)$ running over the set $\Vc_n$, the forms
 \begin{align*}
Tc_I(\hat\om_\nb) \wg c_J (\hat{\Om}_\nb) := 
Tc_{i_1}( \hat{\om}_\nb)\wg \ldots \wg Tc_{i_p}(\hat{\om}_\nb) \wg  
c_{j_1}(\hat{\Om}_\nb) \wg  \ldots \wg c_{j_q} (\hat{\Om}_\nb)  
\end{align*}
are closed and their cohomology classes form
a basis of $H_{\rm d}^{\bullet} (|\triangle_{\Gb} FM|, \Rb)$. 

The cocycles obtained by their integration along fibers,
 \begin{align} \label{basis}
C_{I, J} (\nabla)\, := \circint_{\D^\bullet} Tc_I(\hat\om_\nb) \wg c_J (\hat{\Om}_\nb) \, , 
\quad (I,J) \in \Vc_n  \, ,
 \end{align}
provide a complete set of representatives for a basis of $H_{\rm d , \Gb}^{\bullet} (FM, \Rb)$.
 \end{cor}
 
\smallskip 
  
 The same procedure applies to the relative case. The representatives of the even Chern
 classes are still given by the formula \eqref{Chcoc} with $k = 2i$, while
 the odd Chern forms can be transgressed as follows (cf. \cite[Prop. 5]{Guel}). 
 Denote 
 by $ \Fg \Fl_n  = \Fs_n \oplus  \Fo_n$ the standard decomposition into symmetric and 
 skew-symmetric parts, and let \, $\Fs : \Fg \Fl_n \ra \Fs_n$, resp.  $\Fo : \Fg \Fl_n \ra \Fo_n$,
be the corresponding projections. Then
 \begin{align} \label{transodd}
 \begin{split}
&c_{ 2k-1} (\hat\Om_\nb)  = d (Tc_{ 2k-1}(\hat\om_\nb)), \quad \text{with} \\
&Tc_{ 2k-1}(\hat\om_\nb) = 
(2k-1) \int_0^1 c_{ 2k-1}\left(\Fs(\hat{\om}_\nb),
 \hat{\Om}_t, \ldots , \hat{\Om}_t \right)  dt , \\
&\text{where} \quad \hat{\Om}_t = 
t \Fs(\hat{\Om}_\nb) + \Fo (\hat{\Om}_\nb) + (t^2 - 1) \Fs(\hat{\om}_\nb) \wg \Fs(\hat{\om}_\nb) .
\end{split}
 \end{align}
 
To construct a Vey basis, one now takes the collection $\Vc O_n$ of all pairs $(I, J)$
 of (possibly empty) subsets of $\{1, \ldots , n\}$, with
 $I = \{i_1 < \ldots < i_p \}$ containing only odd integers 
  and $J=\{j_1 \leq \ldots \leq j_q\}$, with  $|J| \leq n$, such that
  $\,   i_0 \leq j_0$ and $\,  i_0 +  |J| > n$. Here $i_0= i_1$ if $I \neq \emptyset$ or
  $i_0 = \ify$ otherwise, and $j_0$ stands for the smallest odd integer in $J$ or
  $j_0 = \ify$ if there is none.

\begin{cor} \label{VObasis}
With $(I, J)$ running over the set $\Vc O_n$, the forms
 \begin{align*}
Tc_I(\hat\om_\nb) \wg c_J (\hat{\Om}_\nb) := 
Tc_{i_1}( \hat{\om}_\nb)&\wg \ldots \wg Tc_{i_p}(\hat{\om}_\nb) 
 \wg c_{j_1}(\hat{\Om}_\nb) \wg  \ldots \wg c_{j_q} (\hat{\Om}_\nb)   
\end{align*}
are closed and their classes form
a basis of $H_{\rm d}^{\bullet} (|\triangle_{\Gb}PM|, \Rb)$. 

The cocycles obtained by their integration along fibers,
 \begin{align} \label{basisO}
C_{I, J} (\nabla)\, := \circint_{\D^\bullet} Tc_I(\hat\om_\nb) \wg c_J (\hat{\Om}_\nb) \, , 
\quad (I,J) \in \Vc O_n  \, ,
 \end{align}
form a complete set of representatives for a basis of $H_{\rm d , \Gb}^{\bullet} (PM, \Rb)$.
 \end{cor}
  \smallskip
   
 In particular, the representatives of the Chern classes are the cocycles
 $C_{\emptyset, \{k\}}(\nabla) \equiv C_k (\hat{\Om}_\nb)$, with $k$ even.

%%%%%%%%%%%%%%%%%%%%%%%%%%%%%%%%%%%%%
 
   \section{Cyclic cohomological models for the Hopf algebra $\Hc_n$}  \label{Hn}

For the convenience of the reader, we collect here a modicum of salient facts 
from \cite{CM98, CM99, MR09, MR11} 
about the Hopf algebra $\Hc_n $ and the Hopf cyclic
cohomological models associated to it.

\subsection{Canonical representation and the standard cyclic model} \label{Canmod}

The Hopf algebra $\Hc_n $ serves as a ``quantum'' analogue of the structure 
group of the universal ``space of leaves'' for codimension $n$ foliations. As 
such, it arises naturally as the symmetry structure of the convolution algebra 
$C_c^{\ify}  (\bar{\G}_n)$ of the \'etale groupoid $\bar{\G}_n$ of germs of 
local diffeomorphisms of $\Rb^n$ acting by prolongation on the frame bundle $F\Rb^n$.
For the clarity of the exposition it is convenient to replace $C_c^{\ify}  (\bar{\G}_n)$ by 
the crossed product algebra $ {\Ac} = C_c^{\ify} (F\Rb^n ) \rtimes \Gb$, where
$\Gb = \Diff {\Rb}^n$ is treated as a discrete group.

In order to implement the operational construction of $\Hc_n$, one
identifies $F\Rb^n$ with the affine group $G = \Rb^n \times \GL_n$ 
in the obvious way, and one
endows it with the canonical form  $\t = ( \t^k) = ({\bf y}^{-1} \, dx)^k$ and with the
 flat connection   $\, \om = ( \om^i_j ) = ({\bf y}^{-1} \, d{\bf y})^i_j \,$. 
 With the usual summation convention,
 the basic horizontal vector fields are
 $X_k = y_k^{\mu} \part_{\mu}$, and the fundamental vertical vector fields are
 $Y_i^j = y_i^{\mu} \part_{\mu}^j$, $i, j, k =1, \ldots, n$,
 where $\displaystyle \part_{\mu} = \frac{\part} {\part x^{\mu}}, \,
\part_{\mu}^j = \frac{\part}{\part y_j^{\mu}}$ .
 The collection $\{ X_k, Y_i^j \}$ forms the standard
 basis  of the Lie algebra $\Fg$ of left-invariant vector fields on  $G$.
 
 The group $\Gb$ acts on $F\Rb^n$ by prolongation,
\begin{equation} \label{frameact}
\phi (x, {\bf y}) := \left( \phi (x), {\phi}^{\prime} (x) \cdot
{\bf y} \right) \, , \quad \text{where} \quad {\phi}^{\prime}(x)^{i}_{j} := 
\part_{j} \, {\phi}^{i}  (x) \, .
\end{equation}
The algebra $\Ac$ can be regarded as the subalgebra of the
endomorphism algebra $\End_{\Cb}\big(C_c^{\ify} (F\Rb^n )\big) $ generated by the
multiplication and the translation operators
\begin{equation*}
M_f (\xi)  = f \, \xi \, , \quad  U_{\vp}^* (\xi) = \xi \circ \vp  \, , \quad
  f, \, \xi \in C_c^{\ify} (F\Rb^n ) , \, \, \vp \in \Gb.
\end{equation*}
Letting the vector fields $Z \in \Fg$ act on $\Ac$ by 
\begin{equation*}
Z (f \, U_{\vp}^*) \, = \,Z( f) \, U_{\vp}^* , \qquad f
\, U_{\vp}^* \in {\Ac} \, ,
\end{equation*}
the resulting operators in $\Lc(\Ac)$ satisfy generalized Leibnitz rules, which in the 
Sweedler notation take the form
\begin{equation} \label{Zrule}
Z (a \, b) \, = \, \sum_{(Z)} Z_{(1)} (a) \, Z_{(2)}(b)  ,
\quad a, b \in {\Ac} .
\end{equation}
In particular, $X_k  \in \Lc ({\Ac}) $ satisfy
\begin{equation*} 
X_k(a \, b) \, = \, X_k (a) \, b \, + \, a \,X_k (b) \, + \,
\d_{jk}^i (a) \, Y_{i}^{j} (b) \, ,
\end{equation*}
where  \, $\d_{jk}^i (f \, U_{\phi}^*) \, =\, \g_{jk}^i  (\phi) \, f \, U_{\vp}^*$, with
\begin{equation} \label{gijk}
\g_{jk}^i (\phi) (x, {\bf y}) \, =\, \left( {\bf y}^{-1} \cdot
{\phi}^{\prime} (x)^{-1} \cdot \part_{\mu} {\phi}^{\prime} (x) \cdot
{\bf y}\right)^i_j \, {\bf y}^{\mu}_k \, ;
\end{equation}
The operators $\, \d_{jk}^i \in \Lc(\Ac)$ are derivations, but their
successive commutators with
the $X_{\ell}$'s yield operators
$\, \d_{jk \,  \ell_1 \ldots \ell_r}^i = [X_{\ell_r} , \ldots
[X_{\ell_1} , \d_{jk}^i] \ldots ]$,
 which involve multiplication by higher order jets of diffeomorphisms
\begin{align} \label{highg}
\begin{split} 
&\d_{jk \,  \ell_1 \ldots \ell_r}^i \, ( f \, U_{\phi}^*) =
 \g_{jk \,  \ell_1 \ldots \ell_r}^i (\phi)\, f \, U_{\phi}^* \qquad   \text{where}\\ 
 & \g_{jk \,  \ell_1 \ldots \ell_r}^i (\phi) =
 X_{\ell_r} \cdots X_{\ell_1} \big(\g_{jk}^i (\phi)\big)  \, , \quad \phi \in \Gb ,
 \end{split}
\end{align}
and satisfy progressively more elaborated Leibnitz rules.
 
The subspace $\Fh_n$ of $\Lc (\Ac)$ generated by the operators
$X_k$, $Y^i_j$, and  $\d_{jk \,  \ell_1  \ldots \ell_r}^i$
forms a Lie algebra $\Fh_n$. By definition, 
$\Hc_n$ is the algebra of operators in $\, \Lc (\Ac)$
generated by $\Fh_n$ and the scalars.
For $n >1$ the operators $\, \d_{jk \,  \ell_1 \ldots \ell_r}^i \,$
are not algebraically independent. They are subject to the ``structure identities''
\begin{equation}  \label{bianchi}
 \d_{j \ell \,  k}^i \, - \,  \d_{j k \,  \ell}^i \, = \,
 \d_{j k}^s \, \d_{s \ell}^i  \ - \, \d_{j \ell}^s \, \d_{s k}^i \, ,
\end{equation}
reflecting the flatness of the standard connection.

The algebra $\Hc_n$ is isomorphic to the quotient 
$\FA ({\Fh}_n)/\Ic$ of the universal enveloping
algebra $\FA ({\Fh}_n)$ by the ideal $\Ic$ generated by the above identities. It possesses
a distinguished character  $\d :\Hc_n \ra \Cb$, which 
 extends the modular character of
 ${\Fg \Fl}_n (\Rb)$, and is induced from the
 character of ${\Fh}_n$ defined by
 \begin{equation} \label{chard}
\d(Y_i^j) = \d_i^j , \quad  \d(X_k) = 0, \quad \d(\d_{jk \,  \ell_1 \ldots \ell_r}^i) =0 .
 \end{equation}

As coalgebras, $\Hc_n$ and $\FA ({\Fh}_n)$ differ drastically however.
The coproduct of $\Hc_n$  
stems from the interplay between the action of $\Hc_n$ and the
product in $\Ac$. More precisely, it confers to $ \Hc_n $ the only
Hopf algebra structure for which
$ \Ac $ is a left $\, \Hc_n $-module algebra. Concretely,
the formula (\ref{Zrule})  extends to all $h \in \Hc_n$,
\begin{equation} \label{HA}
 h (a  b) \ = \ \sum_{(h)} \, h_{(1)} (a) \, h_{(2)} (b) \, ,
\quad  h_{(1)}, h_{(2)} \in \Hc_n , \quad a, b \in \Ac \, ,
\end{equation}
and this uniquely determines a
{\it coproduct} $\D : \Hc \ra \Hc_n \ot \Hc_n$, by
\begin{equation*} 
\D ( h) \ = \ \sum_{(h)} \, h_{(1)} \ot h_{(2)} .
\end{equation*}
 The {\it counit} is \, $ \ve( h)  \, = \,  h (1)$, and there is a canonical
 {\it twisted antipode} determined by the canonical trace of
the crossed product algebra $ {\Ac}$, namely
 \begin{equation} \label{tr}
\tau \, (f \, U_{\vp}^* )  \, = \, \left\{ \begin{matrix}
\displaystyle
  \int_{F\Rb^n} \, f  \, \varpi \, , \quad \text{if}
\quad \vp = Id \, , \cr\cr \quad 0 \, , \qquad \qquad
\text{otherwise} \, ;
\end{matrix} \right.
\end{equation}
here $\varpi$  is the volume form attached to the canonical framing
given by the flat connection
$\varpi \, = \, \bigwedge_{k=1}^n \t^k \wedge \bigwedge_{(i, j)}
\om^i_j $ \, (ordered lexicographically).  
This trace is $\d$-invariant with respect to the action of $\Hc_n$, that is
\begin{equation} \label{it}
\tau (h(a)) \, = \,  \d(h)\, \tau(a) , \qquad h \in \Hc_n, \,  \, a \in \Ac\, .
\end{equation}
The Leibnitz rule \eqref{HA} together with the fact that the pairing $(a, b) \mapsto \tau(a\, b)$ 
 is non-degenerate also ensure the existence and uniqueness of
an anti-automorphism $ S_\d : \Hc_n \ra \Hc_n$,  $S_\d^2 \, = \,\Id$, 
 satisfying  
\begin{equation} \label{sit}
\tau (h(a) \, b) \, = \, \tau \, (a \, S_\d (h) (b)) , \quad h \in \Hc_n \, , \, a, b  \in \Ac ,
\end{equation}
as well as the involutive property
\begin{equation} \label{invant}
S_\d^2 \, = \, \Id .
\end{equation}
Finally, the {\it antipode} of $\Hc_n$ is $S = \check{\d} \ast S_\d $, where 
$\check{\d} $ is the convolution inverse of $\d$.
\smallskip

Using \eqref{sit},
the standard Hopf cyclic model for $\Hc_n$ is ``imported'' from the standard cyclic
model of the algebra $\Ac$, via the characteristic map
\begin{align} \label{char-map}
\begin{split} 
&h^1 \ot \ldots \ot h^q \in \Hc_n^{\ot^q}  \longmapsto
 \chi_{\tau} (h^1 \ot \ldots \ot h^q) \in C^q (\Ac)
\, , \\  
&\chi_{\tau} (h^1 \ot \ldots \ot h^q) (a^0 , \ldots , a^q) = 
 \tau (a^0  h^1 (a^1) \ldots h^q (a^q)) , \quad a^j \in \Ac .
\end{split}
\end{align}
This map is faithul and gives rise to a cyclic structure~\cite{Cext} on 
$\, \{ C^q (\Hc_n ; \delta) := \Hc_n^{\ot^q} \}_{q \geq 0}$, with faces, degeneracies and cyclic
 operator given by
\begin{eqnarray*} 
\d_0 (h^1 \ot \ldots \ot h^{q-1}) &=& 1 \ot h^1
\ot \ldots \ot h^{q-1} , \\  
\d_j (h^1 \ot \ldots \ot h^{q-1}) &=& h^1 \ot \ldots \ot \D h^j \ot
\ldots \ot h^{q-1}, \quad 1 \leq j \leq q-1 , \\  
\d_n (h^1 \ot \ldots \ot h^{q-1}) &=& h^1 \ot \ldots \ot h^{q-1}
\ot 1; \\
 \s_i (h^1 \ot \ldots \ot h^{q+1}) &=& h^1 \ot \ldots \ot \ve
(h^{i+1}) \ot \ldots \ot h^{q+1} , \quad 0 \leq i \leq q \, ; \\  
   \tau_q (h^1 \ot \ldots \ot h^q) &=& S_\d (h^1) \cdot (h^2
\ot \ldots \ot h^q \ot 1) \, .
\end{eqnarray*}
The cyclicity condition $\, \tau_q^{q+1} = \Id \,$ is satisfied precisely
 because of the involutive property \eqref{invant}, to which is actually equivalent.

The periodic Hopf cyclic cohomology $HP^\bullet (\Hc_n; \Cb_\d)$
of $\Hc_n$ with coefficients in the modular pair $(\d, 1)$ is, by definition
(cf. \cite{CM98, CM99}), the $\Zb_2$-graded cohomology of the total complex 
$\, CC^{\rm tot \bullet} (\Hc_n ; \Cb_\d)$ associated to
the bicomplex 
$\, \{C C^{*, *} (\Hc_n ;\Cb_\d), \, b , \, B \}$, where
\begin{equation*}  
CC^{p, q} (\Hc_n ; \Cb_\d) \, = \, \left\{
\begin{matrix}  C^{q-p} (\Hc_n; \Cb_\d) \, , \quad q \geq p \, , \cr
  0 \, ,  \quad \qquad  \qquad q <  p \, , \end{matrix}   \right.
\end{equation*}
 \begin{align*}  
b = \sum_{k=0}^{q+1} (-1)^k \d_k , \qquad
  B = ( \sum_{k=0}^q (-1)^{q \, k}\tau_q^k ) \sigma_{q-1}\tau_q  .
\end{align*} 

To define the periodic Hopf cyclic cohomology of $\Hc_n$ relative to $\O_n$,
one considers the quotient
 $\Qc_n = \Hc_n \ot_{\Uc (\Fo_n)} \Cb \equiv \Hc_n / \Hc_n \Uc^+ (\Fo_n)$, which is
 an $\Hc_n$-module coalgebra with respect to the coproduct and counit
 inherited from $\Hc_n$.
Then $\{ C^q  (\Hc_n, \O_n ; \Cb_\d) : = \, \big(\Qc_n^{\ot q}\big)^{\O_n} \}_{q \geq 0}$,
 is endowed with a cyclic structure given by restricting to $\O_n$-invariants the operators
 \begin{eqnarray*}  
 \d_0 (c^1 \ot \ldots \ot c^{q-1}) &=& 
  \dot{1} \ot c^1 \ot \ldots \ot
  \ldots \ot c^{q-1} ,   \\  
\d_i  (c^1 \ot \ldots \ot c^{q-1})  &=& 
c^1 \ot \ldots \ot \D c^i \ot \ldots \ot c^{q-1} , \quad 1 \leq i \leq q-1; \\ 
\d_n  (c^1 \ot \ldots \ot c^{q-1})  &=& c^1 \ot \ldots \ot c^{q-1}
\ot \dot{1}\, ; \\  
 \s_i  (c^1 \ot \ldots \ot c^{q+1})  &=& c^1 \ot \ldots \ot \ve
(c^{i+1}) \ot \ldots \ot c^{q+1} ,\quad 0 \leq i \leq q \, ; \\   
   \tau_q ( \dot{h}^1 \ot c^2 \ot \ldots \ot c^q) &=& 
    {S_\d}(h^1) \cdot (c^2
   \ot \ldots \ot c^q \ot \dot{1}).
\end{eqnarray*}
The resulting periodic cyclic cohomology is denoted 
$HP^\bullet (\Hc_n, \O_n ; \Cb_\d)$.
 
 \subsection{Bicrossed product and Chevalley-Eilenberg cyclic model} \label{bcp}
 
The Hopf algebra $\Hc_n$ can be reconstructed as bicrossed product of a 
 matched pair of Hopf algebras of classical type (cf. \cite{CM98, MR09}).
 This structure arises naturally
 from the canonical splitting of the group ${\Gb}$ as a set-theoretical
  product $ {\Gb} = G \cdot \Nbo$
 of the group  $G$ of  affine motions of $\Rb^n$ and the group 
 \begin{equation*} 
\Nbo = \{ \psi \in {\Gb} ; \quad \psi (0) = 0, \, \,  \psi' (0) = \Id \} .
   \end{equation*}
If $\phi \in {\Gb} $ and $\, \phi_0 := \phi - \phi (0)$, then its canonical decomposition
is
 \begin{equation} \label{Kac1}
\phi \, = \,  \vp \circ \psi \, , \qquad \vp \in G , \, \psi \in \Nbo ,
 \end{equation}
where
 \begin{align}    \label{Kac2}
 \begin{split}
 \vp (x)&= \, \phi'_0 (0) \cdot x \, + \, \phi (0) , \quad x \in \Rb^n \\
  \psi (x)& = \,   \phi'_0 (0)^{-1} \left(\phi (x) - \phi (0)\right) .
   \end{split}
 \end{align}
Reversing the order in the above decomposition one simultaneously
obtains a pair of well-defined operations, one of $\Nbo$ on $G$ and the other of $G$ on
$\Nbo$:
 \begin{equation} \label{actions}
  \psi \circ \vp \, = \,  (\psi \rt \vp) \circ  (\psi \lt \vp) ,
  \qquad \text{for} \quad \vp \in G \quad \text{and} \quad \psi \in \Nbo .
 \end{equation}
 The operation  $\rt$ is a left action of $\Nbo$ on $G$, and $\lt$ is
a right action of $G$ on $\Nbo$.  Via the identification  $\, G \simeq F\Rb^n$, one 
recognizes $\rt$ as being exactly the action by prolongation \eqref{frameact}.
 
\smallskip
 
To reconstruct $\Hc_n$ one actually uses the pronilpotent group of jets
\begin{align*}
\FN \, : = \, \{j_0^\ify (\psi) \, \mid \, \psi \in \Nbo \} ,
\end{align*}
on which  the jet components are regarded as affine coordinates. 
Thus, the algebra  $\Fc$ of regular functions on $\FN$ consists of 
polynomial expression in the coordinates
\begin{equation*}
\a^i_{j {j_1}j_2\dots j_r}(\psi)= \p_{j_r}\dots \p_{{j_1}} \p_j
\psi^i(x)\mid_{x=0} , \, 1\leq i, j, {j_1}, j_2, \dots,  j_r \leq n .
\end{equation*}
Since  $\, \a^i_j (\psi) = \d^i_j$,  
and for $r \geq 1$ the coefficients $\a^i_{j {j_1}j_2\dots
j_r} (\psi)$ are symmetric in the lower indices but otherwise
arbitrary,  $\Fc$ can be viewed as the free commutative
algebra over $\Cb$ generated by the indeterminates $\{\a^i_{j
{j_1}j_2\dots j_r} ; \, 1 \leq j < {j_1} <j_2 < \dots <  j_r \leq n \}$.
\smallskip

The algebra $\Fc$ inherits from the group $\FN$ a canonical
Hopf algebra structure, in the standard fashion,
with the coproduct $\, \D : \Fc \ra \Fc \ot \Fc $, the antipode 
$\, S : \Fc \ra \Fc$, and the counit $\ve : \Fc \ra \Cb$ determined
by  
\begin{eqnarray*}  
\D(f)(\psi_1, \psi_2) &=& f(\psi_1 \circ\psi_2) , \qquad 
\psi_1, \psi_2 \in N , \\ \notag 
S(f)(\psi)  &=& f(\psi^{-1}) ,
\qquad  \psi \in N, \quad f \in \Fc , \\ \notag
 \ve (f) &=& f(e) .
\end{eqnarray*} 
The coefficients of the Taylor expansion at $e \in G$ of the prolongation of $\psi \in \Nbo$,
\begin{equation} \label{fcoord}
 \eta^i_{j k \ell_1 \ldots \ell_r} (\psi) =  \g^i_{j k \ell_1 \ldots
\ell_r} (\psi)(e) ,  
 \end{equation}
are easily seen to be regular functions on $\FN$, which also generate $\Fc$ as an algebra.
 Letting $ {\Hc_{\rm ab}}$ denote the (commutative) Hopf subalgebra of
$\Hc_n$ generated by the operators
 $\{ \d^i_{j k \ell_1 \ldots \ell_r} ; \, 1\leq i, j,k, \ell_1,  \dots,  \ell_r \leq n \}$,
 one proves using the structure identities \eqref{bianchi}
  that the assignment $ \etabar : \, {\Hc_{\rm ab}}^{\rm cop} \ra \Fc^{\rm cop}$,
\begin{equation} \label{iota}
\etabar (\d^i_{j k \ell_1 \ldots \ell_r}) \, = \, \eta^i_{j k \ell_1
\ldots \ell_r} , \qquad \fl \, 1\leq i, j,k, \ell_1,  \dots,  \ell_r
\leq n \ .
\end{equation} 
defines an isomorphism of Hopf algebras.

 %%%%%%%%%%%%%%%%%%%%%%%%%%%%
  With $\Fg $ denoting the Lie algebra of $G$,
  let  $\Uc := \Uc(\Fg)$ be its universal enveloping algebra.
  The right action $\lt $ of $G$ on $\Nbo$ induces an action of $\Fg$ on $\Fc$,
 \begin{eqnarray}\label{u>f}
(X \rt f)(\psi) = \frac{d}{dt}\mid_{t=0} f (\psi \lt \exp tX) ,
\quad f \in \Fc , \, X \in \Fg.
\end{eqnarray}
 and hence a left action $\rt $ of $\Uc$ on $ \Fc$. Explicitly,
 for any $u \in \Uc $,
 \begin{equation}  \label{giota}
(u \rt \eta^i_{j k \ell_1 \ldots \ell_r}) (\psi)\, = \,u
\big(\g^i_{j k \ell_1 \ldots \ell_r} (\psi)\big)(e) , \qquad \psi
\in \Nbo .
\end{equation}
The right hand side of \eqref{giota}, before evaluation at  $e \in
G$, describes the effect of the action of $u \in \Uc $ on
$\d^i_{j k \ell_1 \ldots \ell_r} \in  {\Hc_{\rm ab}}$. From the very definition 
\eqref{iota}, it follows that $\, \etabar :  {\Hc_{\rm ab}} \ra \Fc$
  identifies the  $\,\Uc $-module  $ {\Hc_{\rm ab}}$
  with the $\,\Uc $-module $\Fc$. In particular $\, \Fc$ is
  $\,\Uc $-module algebra.
 \smallskip

On the other hand, $\Uc $ carries a natural right $\Fc$-comodule structure
$\,\Db:\Uc  \ra\Uc  \ot \Fc$, which can be suggestively described by assigning to each
element $\, u \in \Uc $ a function from $\FN$ to $\Uc $ defined by
   \begin{equation} \label{comod3}
({\Db} u)  (\psi) \, = \,  \tilde{u} (\psi) (e) , \quad \text{where}
\quad
 \tilde{u} (\psi) \, = \,  U_\psi \, u\, U^{\ast}_\psi .
 \end{equation}
This coaction actually endows
 $\Uc (\Fg)$ with the structure of a right $\Fc$-comodule coalgebra.

Thus equipped, $\Uc$ and $\Fc$ form a matched pair
of Hopf algebras, i.e. (with the usual conventions of notation)
satisfy the compatibility conditions
 \begin{align} \notag
&\epsilon(u\rt f)=\epsilon(u)\epsilon(f), \qquad u\in\Uc, \,  f\in \Fc\\  \notag
&\Delta(u\rt f)=u\ps{1}\ns{0} \rt f\ps{1}\ot
u\ps{1}\ns{1}(u\ps{2}\rt f\ps{2}), \\   \notag
 &\Db(1)=1\ot
1, \\  \notag
 &\Db(uv)=u\ps{1}\ns{0} v\ns{0}\ot
u\ps{1}\ns{1}(u\ps{2}\rt v\ns{1}),\\   \notag
 &u\ps{2}\ns{0}\ot
(u\ps{1}\rt f)u\ps{2}\ns{1}=u\ps{1}\ns{0}\ot
u\ps{1}\ns{1}(u\ps{2}\rt f).
\end{align}
One can then form the bicrossed product Hopf algebra $\Fc\acl \Uc$,  
which has the crossed coproduct
 $\Fc\cl \Uc$ as underlying coalgebra, the crossed product
$\Fc\al \Uc$ as underlying algebra, and whose the antipode is given by
\begin{equation} \notag
S(f\acl u)=(1\acl S(u\ns{0}))(S(fu\ns{1})\acl 1) .
\end{equation}
The reconstruction of $\Hc_n$ is now made precise by
the statement that the Hopf algebras $\Fc\acl \Uc$ and ${\Hc_n}^{\rm cop}$ are
canonically isomorphic~\cite[Thm. 2.15]{MR09}, via the identification which at the level of vector
spaces can be described as 
$\, \etabar^{-1} \ot \Id_\Uc : \Fc\acl \Uc \ra {\Hc_n}^{\rm cop}$.

 \bigskip
 
%%%%%%%%%%%%%%%%%%%

Exploiting the bicrossed product structure, and taking advantage of the 
extended framework for Hopf cyclic cohomology with coefficients~\cite{hkrs},
the complex $\, CC^{\rm tot \bullet} (\Hc_n ; \Cb_\d)$ can be replaced (cf.
\cite{MR09, MR11}) by quasi-isomorphic  
bi-cyclic complexes. The latter amalgamate
two classical types of cohomological constructs,
Lie algebra cohomology with coefficients and coalgebra cohomology
with coefficients, with the essential distinction though
that the coefficients are not only acted upon but also `act back'.
  
The first such bicomplex $C^{\bullet, \bullet}(\wg \Fg^\ast, \bigotimes\Fc)$ 
is described by the diagram
\begin{align} \label{C3}
\begin{xy} \xymatrix{  \vdots & \vdots
 &\vdots &&\\
 \wdg^2\Fg^\ast  \ar[u]^{\p_{\Fg}}\ar@<.6 ex>[r]^{b_\Fc~~~~~~~}& \ar@<.6 ex>[l]^{~~B_\Fc} \wdg^2\Fg^\ast\ot\Fc \ar[u]^{\p_{\Fg}} \ar@<.6 ex>[r]^{b_\Fc}& \ar@<.6 ex>[l]^{~~B_\Fc}\wdg^2\Fg^\ast\ot\Fc^{\ot 2} \ar[u]^{\p_{\Fg}} \ar@<.6 ex>[r]^{~~~~~~~~~b_\Fc} & \ar@<.6 ex>[l]^{~~B_\Fc}\hdots&  \\
 \Fg^\ast  \ar[u]^{\p_{\Fg}}\ar@<.6 ex>[r]^{b_\Fc~~~~~}& \ar@<.6 ex>[l]^{~~B_\Fc} \Fg^\ast\ot\Fc \ar[u]^{\p_{\Fg}} \ar@<.6 ex>[r]^{b_\Fc}& \ar@<.6 ex>[l]^{~~B_\Fc} \Fg^\ast\ot \Fc^{\ot 2} \ar[u]^{\p_{\Fg}} \ar@<.6 ex>[r]^{~~~~~b_\Fc }&\ar@<.6 ex>[l]^{~~B_\Fc} \hdots&  \\
 \Cb  \ar[u]^{\p_{\Fg}}\ar@<.6 ex>[r]^{b_\Fc~~~~~~~}& \ar@<.6 ex>[l]^{~~B_\Fc} \Cb\ot \Fc \ar[u]^{\p_{\Fg}}\ar@<.6 ex>[r]^{b_\Fc}& \ar@<.6 ex>[l]^{~~B_\Fc} \Cb\ot \Fc^{\ot 2} \ar[u]^{\p_{\Fg}} \ar@<.6 ex>[r]^{~~~~~b_\Fc} &\ar@<.6 ex>[l]^{~~B_\Fc} \hdots& ;}
\end{xy}
\end{align}
 the coboundary $\p_\Fg $ involves the action of $\Fg$ on the 
coefficients $\Cb_\d\ot\Fc^{\ot q}$,
while $b_\Fc$ and $B_\Fc$ involve the coaction $\Db_\Fg$. 
More precisely,
\begin{align*} 
\begin{split}
&b_\Fc(\one\ot\a\ot f^1\odots f^q)\, =\, \one \ot\a\ot  1\ot f^1\odots f^q\\
&\qquad + \sum_{1\ge i\ge q} (-1)^i \one \ot\a\ot  f^1\odots \D(f^i)\odots f^q\\
&\qquad + (-1)^{q+1}\one\ot\a\ns{0}\ot   f^1\odots f^q\ot  S(\a\ns{1})  ; \\
&B_\Fc= \big(\sum_{i=0}^{q-1}(-1)^{(q-1)i}\tau_\Fc^{i}\big) \s \tau_\Fc  (\Id - (-1)^q \tau_\Fc), 
\qquad \text{with}  \\  
&\tau_\Fc(\one\ot \a\ot f^1\odots f^q)=\one\ot\a\ns{0}\ot
S(f^1)\cdot(f^2\odots f^q\ot S(\a\ns{1}))\\ 
&\text{and} \quad \s(\one\ot\a\ot f^1\odots f^q)= \ve(f^q)\ot\a\ot f^1\odots f^{q-1} .
\end{split}
\end{align*}

The above bicomplex has a  homogeneous version 
$C^{\bullet, \bullet}_\Fc (\wg \Fg^\ast, \bigotimes\Fc)$, with
 \begin{equation}\label{coinv}
C^{p,q}_{\Fc}(\wg \Fg^\ast, \bigotimes\Fc):= ( \wg^p\Fg^\ast\ot  \Fc^{\ot q+1})^\Fc  ,
\end{equation}
defined as follows. An element  $ \sum \a\ot \td f\in ( \wg^p\Fg^\ast\ot  \Fc^{\ot q+1})^\Fc $
if it satisfies the $\Fc$-coinvariance condition:
\begin{equation*} 
\sum \a\ns{0}\ot \;\td{f}\ot S(\a\ns{1})=\sum \; \a\ot\td{f}\ns{0}\ot \td{f}\ns{1} ;
\end{equation*}
here for  $\td f=f^0\odots f^q$, we have denoted
\begin{align*} 
\td f\ns{0}\ot \td f\ns{1}\, = \, f^0\ps{1}\odots
f^q\ps{1}\ot f^0\ps{2}\cdots f^q\ps{2}.
\end{align*}
The identification between the two complexes 
is made by the isomorphism 
 \begin{align*} 
\begin{split}
 &\Ic:  \wg^p\Fg^\ast\ot\Fc^{\ot q}\xrightarrow{\simeq} (\wg^p\Fg^\ast\ot\Fc^{\ot q+1})^\Fc , \qquad
\qquad  \Ic( \a\ot \td f) \,=\\
 &\a\ns{0}\ot f^1\ps{1}\ot S(f^1\ps{2})f^2\ps{1}\odots S(f^{q-1}\ps{2})f^q\ps{1}  \ot S(\a\ns{1} f^q\ps{2}).
 \end{split}
\end{align*}

\smallskip

A closely related bicomplex replaces the tensor powers of the algebra $\Fc$ with 
homogeneous cochains on the group of jets $\FN$ with values in
 $\wg  \Fg^\ast$, namely $\bar{C}^{\bullet, \bullet}(\FN,\wg \Fg^\ast) $, defined as follows:
 \begin{align}  \label{g*N}
\begin{split}
& \bar{C}^q(\FN,\wedge^p \Fg^\ast) = \{ c:  {\FN}^{q+1} \ra \wedge^p\Fg^\ast  \mid \\
&c (\psi_0 \psi, \dots ,\psi_q \psi)= \,\psi^{-1} \rt c (\psi_0,\dots,\psi_q) , \, \, \fl \, \psi \in \FN \} ,
\end{split}
\end{align} 
with boundary operators $\bar{\p} $ and $\bar{B}$,
\begin{align*} 
\begin{xy} \xymatrix{ \vdots & \vdots
 &\vdots  & \\
\bar{C}^0 \big(\FN , \wdg^2\Fg^\ast \big)
  \ar@<.6 ex>[r]^{\bar{b}} \ar[u]^{\bar{\p}}&\ar@<.6 ex>[l]^{\bar{B}}
\bar{C}^1 \big(\FN , \wdg^2\Fg^{\ast}  \big)
  \ar@<.6 ex>[r]^{\bar{b}} \ar[u]^{\bar{\p}}
 & \ar@<.6 ex>[l]^{\bar{B}}\bar{C}^2 \big(\FN ,  \Fg^{\ast}  \big)
 \ar[u]^{\bar{\p}}  \ar@<.6 ex>[r]^{\bar{b}}& \ar@<.6 ex>[l]^{\bar{B}}\hdots\\
\bar{C}^0 \big(\FN , \Fg \big)
 \ar@<.6 ex>[r]^{\bar{b}} \ar[u]^{\bar{\p}}& \ar@<.6 ex>[l]^{\bar{B}}\bar{C}^1 \big(\FN , \Fg^{\ast}  \big)
\ar@<.6 ex>[r]^{\bar{b}} \ar[u]^{\bar{\p}}& \ar@<.6 ex>[l]^{\bar{B}}\bar{C}^2 \big(\FN ,  \Fg^{\ast}\big)
\ar[u]^{\bar{\p}} \ar@<.6 ex>[r]^{\bar{b}}& \ar@<.6 ex>[l]^{\bar{B}}\hdots\\
\bar{C}^0 \big(\FN , \Cb \big) \ar@<.6 ex>[r]^{\bar{b}}\ar[u]^{\bar{\p}}& \ar@<.6 ex>[l]^{\bar{B}} \bar{C}^1
\big(\FN , \Cb  \big)
 \ar@<.6 ex>[r]^{\bar{b}}\ar[u]^{\bar{\p}}& \ar@<.6 ex>[l]^{\bar{B}}\bar{C}^2 \big(\FN , \Cb \big)\ar[u]^{\bar{\p}}
\ar@<.6 ex>[r]^{\bar{b}}& \ar@<.6 ex>[l]^{\bar{B}} \hdots  }
\end{xy}
\end{align*}
 defined as follows: 
\begin{align*} 
\begin{split}
&{(\bar{\p} c) (\psi_0, \dots, \psi_{q}) =\p c(\psi_0, \dots, \psi_{q}) -
 \sum_k \alpha^k\wdg (Z_k\rt c)(\psi_0, \dots, \psi_{q}) }\\
 &\text{where} \, \{Z_k\} \, \text{and} \, \{\alpha_k\} \, \text{are dual bases of} \,  \Fg \, \text{and} \, \Fg^* , 
 \quad \text{and} \\
 &(Z\rt c)(\psi_0, \dots, \psi_{q})= \sum_i \dt c(\psi_0, \dots, 
 {\psi_{i}\lt\exp(tZ)}, \dots, \psi_{q}) ;\\
&\bar{b} c (\psi_0, \dots, \psi_{q+1}) =\sum_{i=0}^{q+1}(-1)^i
c (\psi_0,\dots,\hat{\psi_i},\dots, \psi_{q+1}) ; \\
& \bar{B}= (\sum_{i=0}^{q-1}(-1)^{(q-1)i}\bar{\tau}^{i}) \bar{\s} \bar{\tau} , \qquad
\bar{\tau} (c)(\psi_0, \dots, \psi_q)= c(\psi_1,  \dots, \psi_q,\psi_0), \\
&\text{and} \qquad \qquad 
\bar{\s}(c)(\psi_0,\dots, \psi_{q-1})=c(\psi_0,\dots, \psi_{q-1}, \psi_{q-1}) .
\end{split}
\end{align*} 
 It is isomorphic to the bicomplex \eqref{coinv} via the chain map
 \begin{align*}  
\begin{split}
& \kappa: C^{\bullet, \bullet}_\Fc (\wg \Fg^\ast, \bigotimes\Fc)
\ra \bar{C}^\bullet (\FN,\wedge^p \Fg^\ast) , \\
&\kappa\big(\sum \a\ot \td f\big) \,(\psi_0,\dots, \psi_q) \,=\,\sum f^0(\psi_0)\dots f^q(\psi_q)\a .
\end{split}
\end{align*}
\smallskip
 
Finally, since $\Fc$ is commutative  one can restrict both sides to the 
quasi-isomorphic subcomplexes of totally antisymmetric cochains
 \begin{align} \label{grp-iso-a}
 \kappa_\wg := \kappa \circ \a_\Fc: C^{\bullet, \bullet}_\Fc (\wg \Fg^\ast, \wg\Fc)
\ra \bar{C}_\wg^\bullet (\FN,\wedge \Fg^\ast) , 
 \end{align}
where $\a_\Fc : C^{\bullet, \bullet}_\Fc (\wg \Fg^\ast, \wg\Fc) \ra
C^{\bullet, \bullet}_\Fc (\wg \Fg^\ast, \bigotimes\Fc) $ is the
antisymmetrization map.
Diagrammatically,
\begin{align*} 
\begin{xy} \xymatrix{  \vdots & \vdots
 &\vdots &&\\
 \wdg^2\Fg^\ast  \ar[u]^{\p_{\wg}}\ar[r]^{b_{\wg}~~~~~~~}&  (\wdg^2\Fg^\ast\ot\wdg^2 \Fc)^\Fc \ar[u]^{\p_{\wg}} \ar[r]^{b_{\wg}}& (\wdg^2\Fg^\ast\ot\wdg^3\Fc)^\Fc \ar[u]^{\p_{\wg}} \ar[r]^{~~~~~~~~~b_{\wg}} & \hdots&  \\
 \Fg^\ast  \ar[u]^{\p_{\wg}}\ar[r]^{b_{\wg}~~~~~}&  (\Fg^\ast\ot\wdg^2 \Fc)^\Fc \ar[u]^{\p_{\wg}} \ar[r]^{b_{\wg}}& (\Fg^\ast\ot\wdg^3\Fc)^\Fc \ar[u]^{\p_{\wg}} \ar[r]^{~~~~~b_{\wg} }& \hdots&  \\
 \Cb  \ar[u]^{\p_{\wg}}\ar[r]^{b_{\wg}~~~~~~~}&  (\Cb\ot\wdg^2 \Fc)^\Fc \ar[u]^{\p_{\wg}} \ar[r]^{b_{\wg}}& (\Cb\ot\wdg^3\Fc)^\Fc \ar[u]^{\p_{\wg}} \ar[r]^{~~~~~b_{\wg}} & \hdots&,  }
\end{xy}
\end{align*}
The boundary operators of the bicomplex 
$C^{\bullet, \bullet}_\Fc (\wg \Fg^\ast, \bigotimes\Fc) $ acquire a simpler form
when restricted to $C^{\bullet, \bullet}_\Fc (\wg \Fg^\ast, \wg\Fc)$.
Thus, $B_\Fc =0$ and the others are given by
 \begin{align*}
\begin{split}
&b_{\wg}(\a\ot f^0\wdots f^q)= \a\ot 1\wdg f^0\wdots f^q \, ; \\
&\p_{\wg}(\a\ot f^0\wdots f^q= \\
&\qquad\p\a\ot f^0\wdots f^q \, -\, \sum_k \alpha^k\wdg \a\ot \wedge\ot Z_k\rt(f^0\wdots f^q) .
\end{split}
\end{align*}
The total cohomology of the above bicomplex is canonically isomorphic to 
$HP^\bullet (\Hc_n ; \Cb_\d)$ and will be denoted by $HP_{\acl}^\bullet(\Hc_n)$.

\smallskip
 
The relative (to $\O_n$) version of the above cohomology will be denoted $HP_{\acl}^\bullet(\Hc_n, \O_n)$.  
It is canonically isomorphic to 
$HP^\bullet (\Hc_n , \O_n; \Cb_\d)$, via the quasi-isomorphism of relative complexes 
obtained by restricting
to $\O_n$-basic cochains on both sides. This amounts to replacing
$\Fg = \Rb^n \rtimes \Fg\Fl_n (\Rb)$ by
$\Fg/\Fo_n $, and then restricting to $\O_n$-invariant
cochains. The group $\O_n$ acts on $\FN$ by the restriction of the right action of $\Gb$.
The chain map $ \kappa_\wg $ induces a quasi-isomorphism
 \begin{align} \label{grp-iso-rel}
 \kappa_\wg^{\O_n}  : C^{\rm tot \bullet}_\Fc (\wg (\Fg/\Fo_n)^\ast, \wg\Fc)^{\O_n}
\ra \bar{C}_\wg^{\rm tot \bullet} (\FN,\wedge (\Fg/\Fo_n)^\ast)^{\O_n} . 
 \end{align}

  \section{Hopf cyclic characteristic cocycles} \label{HCCC}
  
For the transfer of the characteristic cocycles to Hopf cyclic cohomology
we shall use two analogues of the classical van Est isomorphism. 
The first one, recalled below, was established in \cite{CM98} 
and provided the means to 
identify the Hopf cyclic cohomology of $\Hc_n$ with 
 the Gelfand-Fuks cohomology of the Lie algebra $\Fa_n$.
The second one, derived in \S \ref{HtoDC}, identifies
the Hopf cyclic cohomology to the differentiable cohomology.
The transferral proper of the characteristic cocycles is then achieved in \S \ref{TCC}.

  \subsection{From Lie algebra to Hopf cyclic cohomology}  \label{LtoH}

We begin by recalling the first quasi-isomorphism, in the form refined in \cite{MR11}.
The setting is very similar to that
described in \S \ref{DEC}, only here it is specialized to $M=\Rb^n$, endowed with
the standard flat connection, still denoted by $\nabla$. 
As in \S \ref{bcp},
we identify $F\Rb^n$ and the affine group $G$.
\smallskip

Instead of the map $\, \hat{\s} : |\bar\triangle_{\Gb} FM| \ra F^\ify M$ of \eqref{xchng2},
we now consider the map $\, \hat{\vars} : |\bar\triangle_{\Nbo} F\Rb^n| \ra F^\ify \Rb^n$, whose 
simplicial components are defined, in homogeneous group coordinates, by
\begin{align} \label{vars}
\begin{split}
&\vars_p (\tb ; \psi_0 , \ldots , \psi_p, \vp) \, = \,
 \vp \cdot \big(\sb_{(\psi_0,\dots,\psi_p)}  (\tb) \lt \vp \big)^{-1},  \quad  \text{where} \\
& \sb_{(\psi_0,\dots,\psi_p)}  (\tb) \, = \,
    \sum_{i=0}^p t_i \, j_0^\infty(\psi_i) , \quad \vp \in G, \, \, \psi_0,\dots,\psi_p \in \Nbo .
\end{split}
\end{align}
The composition $\, \displaystyle \Dc \, = \, \circint_{\D^\bullet} \circ \hat{\vars}^* $
defines a new map of complexes $ \Dc : C^\bullet (\Fa_n) \ra
 \bar{C}_{\rm d}^{\rm tot \, \bullet} \left(\Nbo, \Om^\ast (FM)\right)$, which
 satisfies the enhanced covariance property 
\begin{align} \label{Dcov}
\Dc  (\om) (\psi_0 \lt \phi,\dots,\psi_p \lt \phi) =
\phi^\ast \left(\Dc (\om) (\psi_0,\dots, \psi_p)\right), \, \, \fl   \phi \in \Gb.
\end{align} 
Taking $\phi \in G$, this relation shows that
 $\Dc (\om)$ is completely determined by its values at the identity $e \in G$.
One is led then to define
  \begin{align} \label{E}
\Ec (\om) (\psi_0,\dots, \psi_p): = \, \Dc (\om)(\psi_0,\dots, \psi_p)
 \mid_{\vp=e} \,
  \in \, \wedge^\bullet \Fg^\ast.
\end{align}
A more explicit expression for $\Ec (\om) $ is obtained as follows.
Fix  a basis $\{\a_k\}$ of $\Fg^*$, and denote by $\{\td\a_k\}$ the corresponding
 left invariant forms on $G$, and by $\{Z^k\}$ be the dual basis 
 of left invariant vector fields. 
Define $\,  \nu(\vp, \psi) : = \, \vp \circ (\psi \lt \vp)^{-1}$, and let $\imath :\Gb \ra \Gb$ 
be the inversion map $\imath(\rho) = \rho^{-1}$. Then
 \begin{align*}  
 \begin{split}
 \Ec (\om) &= \circint_{\D^\bullet}  \hat{\vars}^*\big(\mu (\om) \big) , 
 \quad  \text{where} \quad \mu (\om) \, = \,  \sum_{|I| = r}
\imath^\ast \big( \iota_{Z_e^I} \nu^\ast ({\tilde\om})\big) \ot \a_I ,\\
&  \text{with} \quad I =(i_1 < \ldots < i_r) \quad \text{and} \quad \a_I = \a_{i_1} \wg \ldots \wg \a_{i_r} ,
 \end{split}
 \end{align*}
 or in a more suggestive notation,
 \begin{align}  \label{dupont3}
 \Ec (\om) (\psi_0,\dots, \psi_p) \, = \, \int_{\D (\psi_0,\dots, \psi_p)}  \mu (\om) .
\end{align} 
The way in which $\Dc (\om)$ can be recovered from $\Ec (\om)$ is made 
precise by the following identity:  
\begin{align*} 
 \begin{split}
\Dc  (\om) (\psi_0,\dots, \psi_p) \mid_\vp &= \, 
\int_{\D (\psi_0 \lt \vp,\dots, \psi_p \lt \vp)}  \td{\mu} (\om) \mid_\vp \, , \\
 \quad  \text{where} \quad
 \td{\mu} (\om) &= \sum_{|I| = r} \imath^\ast \big( \iota_{Z_e^I} \nu^\ast ({\tilde\om})\big) \ot \td\a_I .
  \end{split}
\end{align*} 
 Note that the only difference between $\mu(\om)$ and $ \td{\mu} (\om)$ is the replacement
of the $\a_I \in \wedge^\bullet \Fg^\ast$  by the associated left invariant forms
$\td{\a}_I \in \Om^\bullet (G)$. 

Thus, denoting by $L_\vp$ the left translation by $\vp \in G$,
 the above identity can be stated in the equivalent form 
\begin{align} \label{EtoD}
 \Dc  (\om) (\psi_0,\dots, \psi_p) \mid_\vp \, = \,
 L_\vp^\ast \big( \Ec  (\om) (\psi_0 \lt \vp,\dots, \psi_p \lt \vp) \big) .
\end{align}
\smallskip

The first part of the van Est theorem, proved in~\cite{CM98} and in the form stated below
in~\cite{MR11}, can be formulated as follows.

\begin{thm} \label{vE1}
 For any $\om \in C(\Fa_n)$,  
$\Ec (\om) \in C^{\bullet}_\Fc (\wg \Fg^\ast, \wg\Fc)$  
and the resulting map 
$\, \Ec : C^\bullet (\Fa_n) \ra C^{\rm tot \bullet}_\Fc (\wg \Fg^\ast, \wg\Fc)$ 
 is a quasi-isomorphism.
The induced map 
$\, \Ec^{\O_n} : C^\bullet (\Fa_n, \O_n) \ra 
C^{\rm tot \bullet}_\Fc (\wg (\Fg/\Fo_n)^\ast, \wg\Fc)^{\O_n}$ 
is also a quasi-isomorphism.
\end{thm} 

\smallskip
 
The full version of the van Est theorem (cf.~\cite{CM98})
involves the map $\Phi$ of Connes~\cite[III.2.$\d$]{book}, so we
recall its definition specialized to our context.

Consider the DG-algebra,
$\, {\Bc_\Gb} (G)= \Om^*_c(G) \ot \wg \, \Cb [\Gb']$, where $\Gb' = \Gb \setminus \{e\}$,
with the differential $d \ot \Id$.
One labels the generators of $\Cb [\Gb']$ as
$\g_{\phi}$, $\phi \in \Gb$, with $\g_e = 0$, and one forms the crossed product
$\, \Cc_\Gb (G)  = {\Bc_\Gb}(G) \rtimes \Gb$, 
with the commutation rules
\begin{align*}
&U_{\phi}^\ast \, \om \, U_{\phi} = \phi^\ast \, \om  , &\qquad
\, \om \in \Om^*_c(G),\\
& U_{\phi_1}^\ast \, \g_{\phi_2} \, U_{\phi_1} =\g_{\phi_2 \circ \phi_1} -
\g_{\phi_1} , &\qquad  \phi_1 , \phi_2 \in \Gb \, .
\end{align*}
$\Cc_\Gb (G) $ is also a DG-algebra, equipped with the differential  
\begin{equation} \label{dbo}
{\dbo} (b \, U_{\phi}^\ast) = db \, U_{\phi}^\ast - (-1)^{\p b} \, b \, \g_{\phi} \,
U_{\phi}^\ast  , \qquad b \in {\Bc_\Gb} (G) , \quad \phi \in \Gb,
\end{equation}

A cochain $\lambda \in \bar{C}^{q}(\Gb, \Om^p(G))$ determines a linear form
$\wt{\lambda}$ on $\Cc_\Gb (G) $ as follows: 
\begin{align}   \label{prePhi}
\begin{split}
&\wt{\lambda} (b \,U_{\phi}^\ast) = 0 \qquad  \text{for} \quad  \phi \ne \one ; \\
& \text{if}  \quad \phi = \one \quad  \text{and} \quad
 b=\om \ot \g_{\rho_1} \ldots \g_{\rho_q} \qquad  \text{then} \\
&\wt{\lambda}(\om \ot \g_{\rho_1} \ldots \g_{\rho_q}) = \int_{G}
 \lambda(1, \rho_1 , \ldots ,\rho_q) \wg \om .
   \end{split}
\end{align}
The map $\Phi$ from $ \bar{C}^{\bullet}(\Gb, \Om^\bullet(G))$
to the $(b, B)$-complex of the algebra $\Ac = C_c^\ify (G)  \rtimes \Gb$ 
is now defined for  $\lambda \in \bar{C}^{q}(\Gb, \Om^p(G))$ by
\begin{align}   \label{mapPhi}
\begin{split}
\Phi(\lambda)(a^0, \ldots, a^m)&=
\frac{p!}{(m+1)!}
\sum_{j=0}^l(-1)^{j(m-j)}\wt{\lambda}({\dbo}a^{j+1}\cdots {\dbo}a^m\; a^0\; 
{\dbo}a^1\cdots {\dbo}a^j) \\
   \text{where} \quad m &=\dim G-p+q  , \qquad a^0, \ldots, a^m \in \Ac .
   \end{split}
\end{align}
By~\cite[III.2.$\d$, Thm. 14]{book}, $\Phi$ is a chain map to the total
 $(b, B)$-complex of the algebra $\Ac$.
 \smallskip

It is shown in~\cite[pp. 233-234]{CM98}), that if
$\lambda \in \bar{C}^{q}(\Gb, \Om^p(G))$ is of the form $\lambda = \Dc  (\om)$
with  $\om \in C(\Fa_n)$
then $\Phi(\lb)$ has the expression  
\begin{align}  \label{Phim}
 \Phi(\lambda)(a^0, \ldots, a^q) \, = \,
\sum_\a \tau (a^0  h_\a^1 (a^1) \ldots h_\a^q (a^q)) , \qquad h_\a^i \in \Hc_n ;
\end{align}
the tensor $\, \sum_\a  h_\a^1 \ot \ldots \ot h_\a^q \in \Hc_n^{\ot \, q}$ is
uniquely determined, because the characteristic map  \eqref{char-map} is
faithful. Via the corresponding identification, $\Phi(\lb)$ becomes a chain in the 
$(b, B)$-complex which defines the Hopf cyclic cohomology of $ \Hc_n $.
By restricting $\Phi$ to the subcomplex 
\begin{align}  \label{sups}
\bar{C}_{\Dc}^{\rm tot}(\Gb, \Om^*(G)) :=\Dc \big(C(\Fa_n)\big) 
\subset \bar{C}_{\rm d}^{\rm tot}(\Gb, \Om^*(G)) ,
\end{align}
one thus obtains a map
\begin{align}  \label{ups}
\begin{split}
 \Phi_{\rm d}: &\bar{C}_{\Dc}^{\rm tot}(\Gb, \Om^*(G)) \ra CC^{\rm tot \bullet} (\Hc_n, \Cb_\d)\\
 \Phi_{\rm d}(\lb) &= \, \sum_\a  h_\a^1 \ot \ldots \ot h_\a^q \, \in \, \Hc_n^{\ot \, q} .
 \end{split}
\end{align}
$ \Phi_{\rm d}$ is tautologically a chain map, because the Hopf cyclic structure of
$\Hc_n$ was imported from that of the cyclic complex of $\Ac$.
 \smallskip

Furthermore, by restriction to $\O_n$-basic forms on $G$ one obtains
the relative version of the above chain map, which lands in the relative version
of the above Hopf cyclic complex $\, \Hc_n^{\natural}(\O_n; \d) $. 
\smallskip

With this at hand, the van Est type result proved in~\cite[Theorem 11]{CM98} 
can be rephrased as follows.

\begin{thm} \label{vE}
 The composition $\Phi_{\rm d} \circ \Dc : C^\bullet (\Fa_n) \ra CC^{\rm tot *} (\Hc_n; \Cb_\d)$,  together with
its restriction  $ C^\bullet (\Fa_n , \O_n) \ra CC^{\rm tot *} (\Hc_n, O_n; \Cb_\d)$,
are quasi-isomorphisms.
 \end{thm}

\smallskip
  
\subsection{From Hopf cyclic to differentiable cohomology} \label{HtoDC}

In~\cite[\S 3.2]{MR11}  we have constructed another map 
of bicomplexes, $\Theta: C^{\bullet}_\Fc (\wg \Fg^\ast, \wg\Fc) \ra
\bar{C}_{\rm d}^{\bullet}(\Gb, \Om^*(G))$, and we are now in a position
to prove that it too is a quasi-isomorphism.
  
In order to define it, we recall
the isomorphism $\etabar:  \Hc^{\cop}_{\rm ab} \rightarrow \Fc$ of \eqref{iota},
and denote its inverse $\, \dbar = \etabar^{-1}$. Given $f \in \Fc$, one defines the function
$\gbar(f) :\Gb \ra C^\ify (G)$ by
 \begin{align} \label{gbar}
 \dbar(S(f))(U_\phi) \, = \,  \gbar(f)(\phi) \,U^*_\phi ,  \qquad \fl \, \phi \in \Gb .
\end{align}
The left hand side uses the natural action of $\Hc_n$ on the crossed product algebra
$\td\Ac = C^\ify (G) \rtimes \Gb$. 
 The function $ \gbar(f)(\phi) \in C^\ify (G) $,
depends smoothly (in fact algebraically) on the components of the $k$-jet 
of $\phi$, for some $k \in \Nb$.
For example, one can easily see that
\begin{align} \label{bargeta1}
 \gbar(S(\eta^i_{j k}))(\phi^{-1}) \, = \, \g^i_{j k}(\phi) , \qquad \phi \in \Gb .
\end{align} 
 
With this notation, $\, \Theta: C^{\bullet}_\Fc (\wg \Fg^\ast, \wg\Fc) \ra
\bar{C}_{\rm d}^{\bullet}(\Gb, \Om^*(G))$ is given by the formula
\begin{align} \label{Theta}
\begin{split}
\Theta&\big(\sum_{|I|=q} \a_I\ot \lu{I}f^0 \wg \cdots
 \wg \lu{I}f^p \big)(\phi_0, \dots ,\phi_p)=\\
&\sum_{I} \sum_{\s \in S_{p+1}} 
(-1)^\s   \gbar(S( \lu{I}f^{\s(0)}))(\phi_0^{-1})\dots \gbar(S(\lu{I}f^{\s(p)}))(\phi_p^{-1}){{\td\a_I}} ;
\end{split}
\end{align}
here, as in \S \ref{LtoH},
$\td\a_I$ stands for the left-invariant form on $G$ corresponding to $\a_I \in \wg \Fg^*$.
\smallskip

We shall first show that the map $\Theta$ satisfies
a property completely similar 
to that described by formula \eqref{EtoD}. Given the form 
of the expression in  the right hand side of \eqref{Theta}, to justify this
it suffices to prove the following lemma.

\begin{lem} \label{step1}
For any $f \in \Fc$, $\psi \in \Nbo$ and $\vp \in G$, one has
\begin{align} \label{TtoD}
 \gbar(S(f))(\psi^{-1}) (\vp)  =  \gbar(S(f))((\psi \lt \vp)^{-1})(e) .
 \end{align} 
\end{lem}

\begin{proof}  Using the cocycle property of $\, \g^i_{j k}$ and the fact that
  $ \vp \in G$ is affine, one has for any $\psi \in \Nbo$,
  \begin{equation*} 
\g_{jk}^i (\psi \vp)  \, =\,  \g_{jk}^i (\psi) \circ \vp \, +
\, \g_{jk}^i (\vp)  \, =\,   \g_{jk}^i (\psi)\circ \vp .
\end{equation*}
By successive differentiation with respect to left invariant vector fields $X_k$,
one obtains
\begin{align*} 
 \g^i_{j k \ell_1 \ldots \ell_r} (\psi \vp)  \, = \,  \g^i_{j k \ell_1 \ldots \ell_r} (\psi )\circ \vp.
\end{align*}
Letting $e =(0, \one)$ be the base frame we note that, by our identification of $G$
with $F\Rb^n$, $\vp(e) \equiv \vp$. 
Thus, the above identity evaluated at $e$ gives
\begin{align} \label{highgpsi}
\g^i_{j k \ell_1 \ldots \ell_r}  (\psi \vp) (e) \, = \, 
\g^i_{j k \ell_1 \ldots \ell_r}  (\psi)({\vp}(e)) \equiv  \g^i_{j k} (\psi)(\vp) .
\end{align}

Again by the cocycle property, if $\rho \in G$ then
 \begin{equation*} 
\g_{jk}^i (\rho \phi)  \, =\,  \g_{jk}^i (\rho) \circ \phi \, +
\, \g_{jk}^i (\phi)  \, =\,   \g_{jk}^i (\phi) .
\end{equation*}
Therefore, by differentiation,
  \begin{align} \label{ltact1}
   \g^i_{j k \ell_1 \ldots \ell_r} (\rho \phi)\, = \, \g^i_{j k \ell_1 \ldots \ell_r} ( \phi) , \qquad \rho \in G, 
   \quad \phi \in \Gb.
   \end{align}
Writing now $\, \psi \vp \, = \, (\psi \rt \vp) \, (\psi \lt \vp)$, for $\vp \in G$, $\psi \in \Nbo$,
on applying \eqref{ltact1}
one obtains
\begin{align*}
 \g^i_{j k \ell_1 \ldots \ell_r} (\psi \vp)  \, = \,  \g^i_{j k \ell_1 \ldots \ell_r} (\psi \lt  \vp) ,
\end{align*}
and so by \eqref{highgpsi},
\begin{align} \label{fing}
 \g^i_{j k \ell_1 \ldots \ell_r} (\psi) (\vp)  \, = \,  \g^i_{j k \ell_1 \ldots \ell_r} (\psi \lt  \vp)(e) .
\end{align}

On the other hand, the identity \eqref{bargeta1} is valid for higher order jets.
Indeed, applying the definition \eqref{gbar} to $f = \eta^i_{j k \ell}$, one has
\begin{align*}
\begin{split}
& \gbar(S(\eta^i_{j k \ell}))(\phi^{-1})\, U^\ast_\phi = \d^i_{j k \ell} (U^\ast_\phi) = 
[X_\ell ,  \d^i_{j k}] (U^\ast_\phi) = X_\ell \big( \d^i_{j k} (U^\ast_\phi) \big) \\
&=  X_\ell \big( \g^i_{j k}(\phi) \, U^\ast_\phi \big) 
= \g^i_{j k \ell} (\phi)\, U^\ast_\phi.
\end{split}
\end{align*}
Repeated applications give the general identity
\begin{align}  \label{bargeta2}
\gbar(S(\eta^i_{j k \ell_1 \ldots \ell_r}))(\phi^{-1})\, = \, \g^i_{j k \ell_1 \ldots \ell_r} (\phi) .
\end{align}

The relations \eqref{highgpsi} and \eqref{bargeta2} taken together imply
\begin{align*}
\gbar(S(\eta^i_{j k \ell_1 \ldots \ell_r}))(\psi^{-1}) (\vp) \, = \, 
\gbar(S(\eta^i_{j k \ell_1 \ldots \ell_r}))\big((\psi \lt \vp)^{-1}\big) (e) ,
\end{align*}
which proves the statement for a set of generators 
of the algebra $\Fc$.

To complete the proof it
remains to notice that
 \begin{align*}
 \gbar (f_1 f_2) =  \gbar (f_1) \, \gbar (f_2) , \qquad f_1 f_2 \in \Fc ,
\end{align*}
and therefore both sides of the desired relation behave multiplicatively.
\end{proof}

Relying on this lemma, we can now prove a key relation
between the maps of complexes constructed before. 

\begin{lem}  \label{TED}
One has \, $\Theta \circ \Ec  \, = \, \Dc$. 
\end{lem}

\begin{proof} Let $\om \in C^\bullet (\Fa_n)$ and denote
$\varpi = \Ec (\om) \in C^{\bullet}_\Fc (\wg \Fg^\ast, \wg\Fc) $.
By the very definition \eqref{E},
 \begin{align*}
\Dc (\om) (\psi_0,\dots, \psi_p) (e) \, = \, \varpi (\psi_0,\dots, \psi_p) (e)  \in \wg^\bullet \Fg^*,
\end{align*}
while by \eqref{EtoD} on the one hand and
Lemma \ref{step1} on the other, for any $\vp \in G$ one has
 \begin{align*}
\Dc (\om) (\psi_0,\dots, \psi_p) (\vp) = L_\vp^*\big(\varpi (\psi_0,\dots, \psi_p) (e)\big)
= \Theta (\varpi)   (\psi_0,\dots, \psi_p) (\vp) .
\end{align*} 
  \end{proof} 
  
 The next important step is to reconcile the two maps denoted by $\Dc$, that defined
in Theorem \ref{main2} and the one defined in \S~\ref{LtoH}.
 
 \begin{lem} \label{DD}
 With $\nb$ denoting the standard flat linear connection on $\Rb^n$, one has
 \, $\Dc_\nb   \, = \,  \Dc$.   
\end{lem}  
  
 \begin{proof}  
 The construction of the two maps starts with two different
 cross-sections, $\s_\nb : F\Rb^n \ra F^\ify\Rb^n $ defined by \eqref{jcon} and
 $\vars : F\Rb^n \ra F^\ify\Rb^n$ of \eqref{vars}. We need to show that
  they both lead to the same map from $\D_\Gb F\Rb^n$ to $F^\ify\Rb^n $.
 
 Let $u \in F\Rb^n$ be represented as $j_0^1 (\vp) \equiv \rho$,
 with $\rho \in G$. We claim that for any $\psi \in N$,
\begin{align} 
\s_{\nabla^\psi} (u) = \rho \cdot j_0^{\ify}(\psi \lt \rho)^{-1} 
\end{align} 
Indeed with the usual identification $G \cong F\Rb^n$, 
\begin{align*} 
\exp_{\rho(0)}^{\nabla}(u(\xi)) = \exp_{\rho(0)}^{\nabla}(\rho'(0)\xi) =
\rho(0) + \rho'(0)\xi = \rho (\xi) , \qquad \xi \in \Rb^n.
\end{align*} 
Since the left action of $\Nbo$ on $G$ coincides with the natural action on $F\Rb^n$,
$\psi (u)\equiv \psi \rt \rho$. Thus, using
 the naturality property \eqref{nat} one obtains
 \begin{align*}
&  \s_{\nabla^\psi} (u) =  j_0^{\ify} \left(\exp_{\rho(0)}^{\nabla^\psi} \circ u \right) 
=  j_0^{\ify} \left(\psi^{-1} \circ \exp_{\psi(\rho(0))}^{\nabla} \circ \psi (u) \right)  \\
& =j_0^{\ify} \left(\psi^{-1} \circ (\psi \rt \rho) \right) 
=  j_0^{\ify} \left(\psi^{-1} \circ (\psi \rt \rho) \circ (\psi \lt \rho)\circ (\psi \lt \rho)^{-1} \right) \\
&=  j_0^{\ify} \left(\psi^{-1} \circ \psi \circ \rho \circ (\psi \lt \rho)^{-1} \right) 
= \rho \cdot j_0^{\ify}(\psi \lt \rho)^{-1}.
\end{align*}
 This shows that
 \begin{align*}
 \s_p (\tb ; \psi_0 , \ldots , \psi_p, \rho)\,  = \, \vars_p (\tb ; \psi_0 , \ldots , \psi_p, \rho) .
 \end{align*}
 
Now let $\phi \in \Gb$, factorized as the product $\vp \circ \psi$, with $\vp \in G$ and $\psi \in N$.
Then \, $ \s_{\nabla^\phi} \, = \, \s_{\nabla^{\vp \psi}} \, = \, \s_{\nabla^\psi}$ ,
 because   $ \nabla^\vp  =  \nabla$.
 
Therefore, if $\phi_i = \vp_i \psi_i$ , with $\vp_i \in G$ and
 $\psi_i \in \Nbo$, then  for any $\rho \in G$,
 \begin{align} \label{phi2psi}
 \s_p (\tb ; \phi_0 , \ldots , \phi_p, \rho)\,  = \, \vars_p (\tb ; \psi_0 , \ldots , \psi_p, \rho) ,
\end{align}
which completes the proof. 
   \end{proof} 
   
Lemmas \ref{TED} and \ref{DD} taken together ensure that
 \, $\Theta \circ \Ec  \, = \, \Dc_\nb$. Since both $\Dc_\nb$ and $ \Ec $ 
 are quasi-isomorphisms, so must be $\Theta$. This achieves the proof
 of the second explicit analogue of the van Est isomorphism:

\begin{thm} \label{vE2}
  The map  $\Theta: C^{\rm tot \bullet}_\Fc (\wg \Fg^\ast, \wg\Fc) \ra
\bar{C}_{\rm d}^{\rm tot \bullet}(\Gb, \Om^*(G))$ and
the induced map $\Theta^{\O_n}: C^{\rm tot \bullet}_\Fc (\wg (\Fg/\Fo_n)^\ast, \wg\Fc)^{\O_n} \ra
\bar{C}_{\rm d}^{\rm tot \bullet}(\Gb, \Om^*(G/\O_n))$ are quasi-isomorphisms.
 \end{thm}

 \subsection{Hopf cyclic Vey bases} \label{TCC}

As above, $\nb$ stands for the flat connection on
the frame bundle $G \equiv F\Rb^n \ra {\Rb}^n$, 
 with connection form $\om_\nb = \left(\om^i_j \right)$, where
 \begin{equation*} 
 {\om}^i_j \, :=  \, ({\bf y}^{-1})^i_{\mu} \, d{\bf y}^{\mu}_j \, = \, \big({\bf y}^{-1} \, d{\bf y}\big)^i_j
 \, , \qquad i, j =1, \ldots , n \, .
 \end{equation*}
Its  pull-back under the action \eqref{frameact} of $\phi \in \Gb$, is 
the connection 
\begin{align} \label{phicon} 
\phi^* ({\om}^i_j ) \, = \, {\om}^i_j \, +\, \g_{jk}^i (\phi)  \,{\t}^k  .
\end{align}
Thus, the simplicial connection is
\begin{align} \label{scon} 
 \hat{\om}_\nb (\tb ; \phi_0, \ldots , \phi_p)^i_j =  \sum_{r=0}^p t_r \phi_r^* (\om^i_j) 
 = {\om}^i_j + \sum_{r=0}^p t_r  \g_{jk}^i (\phi_r)  \,{\t}^k  
\end{align}
and, taking into account that  \, $ \Om_\nb= 0$, 
the formula  \eqref{scurv} for the simplicial curvature becomes
\begin{align} \label{scurv2}
 \hat{\Om}_\nb(\tb ; \phi_0, \ldots , \phi_p) \, =& \sum_{r=0}^p dt_r \wdg \phi_r^* (\om_\nb)  
- \sum_{r=0}^p t_r \,
\phi_r^* (\om_\nb) \wdg \phi_r^* (\om_\nb) \\ \notag
& + \sum_{r, s=0}^p t_r t_s \,  \phi_r^* (\om_\nb) \wdg  \phi_s^* (\om_\nb) .
 \end{align}

Both $\hat{\om}_\nb$ and $ \hat{\Om}_\nb $ are polynomial forms on $\D_p$
tensored by left $G$-invariant forms $\{ \om^i_j  , \t^k \}$ multiplied
by components of the first-order jet of the prolongation,
 \begin{equation} \label{Gijk}
 \g^i_{j \, k} (\phi) (x, {\bf y}) \, =\, \left( {\bf y}^{-1} \cdot
{\phi}^{\prime} (x)^{-1} \cdot \part_{\mu} {\phi}^{\prime} (x) \cdot
{\bf y}\right)^i_j \, {\bf y}^{\mu}_k \, .
\end{equation}
\smallskip

Recall now that in  \S \ref{CharCo} (see Corollaries \ref{Vbasis} and \ref{VObasis})
we have constructed characteristic cocycles $C_{I, J} (\nb)$ in the differentiable 
Bott bicomplex, and also that Theorem \ref{vE2} provides a quasi-isomorphism
$\Theta: C^{\rm tot \bullet}_\Fc (\wg \Fg^\ast, \wg\Fc) \ra
\bar{C}_{\rm d}^{\rm tot \bullet}(\Gb, \Om^*(G))$.

 \begin{thm} \label{VVOviaT}
The characteristic cocycles $C_{I, J} (\nb)$,
where $(I, J)$ are running over the set $\Vc_n$, resp. $\Vc O_n$, are of the form
\begin{align} \label{sourceCo} 
C_{I, J} (\nb) \, = \, \Theta(\k_{I, J}) \, ;
\end{align}
$\k_{I, J}$ are explicitly defined
cocycles in the bicomplex $C^{\rm tot \bullet}_\Fc (\wg \Fg^\ast, \wg\Fc)$,
and their cohomology classes form a basis of 
$HP_{\acl}^\bullet(\Hc_n)$, resp. $HP_{\acl}^\bullet(\Hc_n, \O_n)$. 
 \end{thm}

\begin{proof}
The cocycles $C_{I, J} (\nb)$ are homogeneous and totally antisymmetric form-valued group cochains in the differentiable Bott bicomplex. Moreover, their values
are combinations of invariant forms
on $G \equiv F\Rb^n$ with coefficients polynomial expressions in $ \g^i_{j \, k} (\phi_r)$'s.
Also, in view of the equality \eqref{phi2psi}, one can restrict the simplicial 
construction to the group $\Nbo$, and thus assume $\phi_r \in \Nbo$.
 
 It then follows from the very definition of the map
\eqref{Theta} together with the identity \eqref{bargeta1} that 
 the preimage of these cochains in the bicomplex 
$ C^{\rm tot \bullet}_\Fc (\wg \Fg^\ast, \wg\Fc)$ is obtained by simply
replacing the $\g^i_{j k}$'s with $\eta^i_{j k}$'s
and the $G$-invariant forms $\td\a_I \in \Om^\bullet (G)$
by their values at the identity, $\a_I \in \wg^\bullet \Fg^*$. 
  
In view of Theorem \ref{vE2}, one obtains this way
a basis of Hopf cyclic characteristic classes.
\end{proof}

\begin{rem}  \label{Rem:VVOviaT}
{\rm From the formulas \eqref{phicon}, \eqref{scon}, \eqref{scurv2} and
 the very definition of the cocycles $\k_{I, J}$, it is clear that their tensor components
are ``economically'' manufactured solely out of elements from
$\wg \Fg^\ast $ tensored by exterior powers of the algebra generated
by $\{ \eta^i_{j k} ; \, 1 \leq i, j, k \leq n\}$. 
This feature constitutes
the analogue of the well-known fact that the Gelfand-Fuks cohomology classes are
representable in terms of $2$-jets.}
\end{rem}
   
\smallskip
 
Returning now to the standard Hopf cyclic cohomological model (see \S \ref{Canmod}),
we recall that the cocycles $C_{I, J} (\nb)$ were obtained by 
transferring Vey bases of $H^* (\Fa_n)$, resp. $H^* (\Fa_n , \O_n)$, via the quasi-isomorphism
$\Dc_\nabla$. Thus, by construction they belong to the subcomplex $\bar{C}_{\Dc}^{\rm tot}(\Gb, \Om^*(G))$
and therefore can be further transported via the quasi-isomorphism \,
 $\Phi_{\rm d}:\bar{C}_{\Dc}^{\rm tot}(\Gb, \Om^*(G)) \ra CC^{\rm tot *} (\Hc_n; \Cb_\d)$
 of \eqref{ups}. Since 
by Theorem \ref{vE}  the composition $\Phi_{\rm d} \circ \Dc_\nabla$ is a quasi-isomorphism,
we conclude that:

 \begin{thm} \label{VVOviaU}
{\rm (1)} The cocycles  $c_{I, J} (\nb) =\Phi_{\rm d}(C_{I, J} (\nb)) $, with $ (I, J) \in \Vc_n$,
form a complete set of representatives for the periodic Hopf cyclic
 cohomology $HP^\bullet (\Hc_n ; \Cb_\d)$.
 
 {\rm (2)} \, The cocycles  $c_{I, J} (\nb)=\Phi^{\O_n}_{\rm d}(C_{I, J} (\nb)) $, with $ (I, J) \in \Vc O_n$,
form a complete set of representatives for the relative periodic Hopf cyclic
 cohomology $HP^\bullet (\Hc_n, \O_n ; \Cb_\d)$. 
 \end{thm}
 
\begin{rem}  \label{Rem:VVOviaU} 
{\em As a final remark, paralleling Remark \ref{Rem:VVOviaT}, we note that  
the cocycles $\Phi_{\rm d}(C_{I, J} (\nb))$ are also ``economically'' constructed. Indeed, 
 their action on tensor products of
 monomials of the form $a = f \, U^*_\phi \in \Ac$ only involves
 the vector fields $X_k, Y^i_j$ applied to the
 function $f \in C_c^\infty (F\Rb^n)$
 and the operators $\d^i_{j k}$ applied to the diffeomorphism $\phi \in \Gb$.  
 The vector fields appear because  
  \begin{align*}
 df \, = \, \sum_{k=1}^n X_k (f) \t^k \, + \, \sum_{i, j=1}^n Y^i_j(f) \, \om^j_i \, ,
 \end{align*}
while the multiplication operators $\{\d^i_{j k}\}$ show up because of the conjugation relation
 \begin{align*}
U_\phi\, df \,U^*_\phi\, = \, \sum_{k=1}^n (X_k (f)\circ \phi)\,  \t^k \, + \, 
\sum_{i, j=1}^n (Y^i_j(f)\circ \phi)  \, ({\om}^i_j \, +\,  \g_{jk}^i (\phi)  \,{\t}^k) .
 \end{align*}
Since the characteristic map \eqref{char-map} is known to be faithful, it follows
that the Hopf cyclic cocycles $c_{I, J} (\nb) \in \sum_{q \geq 0} \Hc_n^{\ot^q}$
(resp. $\sum_{q \geq 0} \left(\Qc_n^{\ot^q}\right)^{\O_n}$)
have their tensorial components made out of the basic generators
 $X_k, Y^i_j$ and $\d^i_{j k}$, and do not involve any $\d^i_{j k \ldots}$
 operators of higher order.}
 \end{rem}
 
 \bigskip
 %%%%%%%%%%%%%%%%%%%%%%%%%%%%%%%%%% 

%%%%%%%%%%%%%%%%%%%%%%%%%%%%%%%%%
 \end{document}